\documentclass[twocolumn]{autart}    

\usepackage{amsmath,amsfonts,stmaryrd,amssymb} 
\usepackage{enumerate} 
\usepackage{enumitem} 
\usepackage{multirow}
\usepackage[justification=centering]{caption}
\usepackage{booktabs}       

\usepackage{cite}
\usepackage{natbib}

\usepackage{amsmath}
\newtheorem{asmp}[thm]{Assumption}

\usepackage{xcolor}

\usepackage{graphicx}
\usepackage{subfigure}

\newcommand{\inner}[2]{\left\langle #1, #2 \right\rangle}

\begin{document}

\begin{frontmatter}

\title{A New One-Point Residual-Feedback Oracle For \\ Black-Box Learning and Control} 

\author[Paestum]{Yan Zhang*}\ead{yan.zhang2@duke.edu},    
\author[Rome]{Yi Zhou*}\ead{yi.zhou@utah.edu},               
\author[Baiae]{Kaiyi Ji}\ead{ji.367@osu.edu},  
\author[Paestum]{Michael M. Zavlanos}\ead{michael.zavlanos@duke.edu}

\address[Paestum]{Mechanical Enginerring and Material Science, Duke University, Durham, NC 27708 USA}  
\address[Rome]{Electrical and Computer Engineering, The University of Utah, Salt Lake City, UT 84112 USA}             
\address[Baiae]{Electrical and Computer Engineering, The Ohio State University, Columbus, OH 43210 USA}        

\thanks[footnoteinfo]{*Equal Contribution. This work is supported in part by AFOSR under award \#FA9550-19-1-0169 and by NSF under award CNS-1932011.}

\begin{abstract}                          
Zeroth-order optimization (ZO) algorithms have been recently used to solve black-box or simulation-based learning and control problems, where the gradient of the objective function cannot be easily computed but can be approximated using the objective function values. 
Many existing ZO algorithms  adopt two-point feedback schemes due to their fast convergence rate compared to one-point feedback schemes. However, two-point schemes require two evaluations of the objective function at each iteration, which can be impractical in applications where the data are not all available a priori, e.g., in online optimization. In this paper, we propose a novel one-point feedback scheme that queries the function value once at each iteration and estimates the gradient using the residual between two consecutive points. When optimizing a deterministic Lipschitz function, we show that the query complexity of ZO with the proposed one-point residual feedback matches that of ZO with the existing two-point schemes. Moreover, the query complexity of the proposed algorithm can be improved when the objective function has Lipschitz gradient. Then, for stochastic bandit optimization problems where only noisy objective function values are given, we show that ZO with one-point residual feedback achieves the same convergence rate as that of two-point scheme with uncontrollable data samples. We demonstrate the effectiveness of the proposed one-point residual feedback via extensive numerical experiments.
\end{abstract}


\end{frontmatter}

\section{Introduction}\label{sec:Intro}
 Zeroth-order optimization algorithms have been widely-used to solve control and machine learning problems where first or second order information (i.e., gradient or Hessian information) is unavailable, e.g., controlling complex systems whose dynamics can not be modeled explicitly but can only be given by high-fidelity simulators \cite{ghadimi2013stochastic}, adversarial training \cite{chen2017zoo}, reinforcement learning \cite{fazel2018global,malik2018derivative} and human-in-the-loop control \cite{luo2020socially}. In these problems, the goal is to solve the following generic optimization problem
 \begin{equation*}
 	\min_{x\in \mathbb{R}^d} f(x), \tag{P}
 \end{equation*}
 where $x \in \mathbb{R}^d$ corresponds to the parameters and $f$ denotes the total loss. Using zeroth-order information, i.e., function evaluations, first-order gradients can be estimated to solve the problem (P). 
 
\begin{table*}[t]
	\caption{Iteration Complexity of Zeroth-order Methods with One-point, Two-point and Proposed Feedback Schemes}
	\label{table:Complexity}
	\setlength\tabcolsep{-0pt} 
		\begin{tabular*}{\textwidth}{@{\extracolsep{\fill}}cccccc}
			\toprule
				Complexity \footnotemark[3] & & Convex $C^{0,0}$ & Convex $C^{1,1}$ &  Nonconvex $C^{0,0}$ & Nonconvex $C^{1,1}$ \\
				\midrule
				\multirow{1}{*}{One-point} & \cite{gasnikov2017stochastic} & $d^2\epsilon^{-4}$ & $d^2\epsilon^{-3}$   & -- & -- \\
				\midrule
				\multirow{4}{*}{Two-point} & \cite{duchi2015optimal} & $d \log(d)\epsilon^{-2}$ & $d \epsilon^{-2} $ & -- & -- \\
				& \cite{shamir2017optimal} & $d\epsilon^{-2}$ & -- & -- & --  \\
				& \cite{nesterov2017random} & $d^2 \epsilon^{-2}$ & $d \epsilon^{-1}$ & $d^3 \epsilon_f^{-1} \epsilon^{-2}$ & $d\epsilon^{-1}$ \\
				& \cite{bach2016highly} & -- & $d^2 \epsilon^{-3}$ (UN) & -- & -- \\
				\midrule
				\multirow{2}{*}{Residual One-point} & Deterministic & $d^2 \epsilon^{-2}$ & $d^{3}\epsilon^{-1.5}$ & $d^4\epsilon_f^{-1} \epsilon^{-2}$ & $d^3 \epsilon^{-1.5}$ \\
				& Stochastic & $d^2 \epsilon^{-4}$ &  $d^2 \epsilon^{-3}$ & $d^3 \epsilon_f^{-3} \epsilon^{-2}$ & $d^4 \epsilon^{-3}$ \\
			\bottomrule
		\end{tabular*}
\end{table*}

Existing zeroth-order optimization (ZO) algorithms
can be divided into two categories, namely, ZO with one-point feedback and ZO with two-point feedback. \cite{flaxman2005online} was among the first to propose a ZO algorithm with one-point feedback, that queries one function value at each iteration to estimate the gradient. The corresponding one-point gradient estimator $\widetilde{\nabla} f(x)$ takes the form \footnote{In \cite{flaxman2005online}, the estimator is $\widetilde{\nabla} f(x) = \frac{du}{\delta} f(x + \delta u)$ where $x \in \mathbb{R}^d$ and $u$ is uniformly sampled from a unit sphere in $\mathbb{R}^d$. In this paper, we follow \cite{nesterov2017random} and sample $u$ from the standard normal distribution.}
\begin{equation}
\label{eqn:OnePoint}
\text{(One-point feedback):}~~\widetilde{\nabla} f(x) = \frac{u}{\delta} f(x + \delta u), 
\end{equation}
where $\delta$ is an exploration parameter and $u\in \mathbb{R}^d$ is sampled from the standard normal distribution element-wise. In particular, \cite{flaxman2005online} showed that the above one-point gradient estimator has a large estimation variance and the resulting ZO algorithm achieves a convergence rate of at most $\mathcal{O}(T^{-\frac{1}{4}})$, where $T$ is the number of iterations, which is much slower than that of gradient descent algorithms used to solve problem (P). Assuming smoothness and relying on self-concordant regularization, \cite{saha2011improved,dekel2015bandit} further improved this convergence speed. However, the gap in the iteration complexity between ZO algorithms with one-point feedback and gradient-based methods remained. In order to reduce the large estimation variance of the above one-point gradient estimator, \cite{agarwal2010optimal,nesterov2017random,shamir2017optimal} introduced the following two-point gradient estimators
\begin{align}
\label{eqn:TwoPoint}
\text{(Two-point} &~ \text{feedback):} \;\; \widetilde{\nabla}f(x) = \frac{u}{\delta} \big(f(x + \delta u) - f(x) \big), \nonumber \\
\text{ or } & \frac{u}{2\delta} \big(f(x + \delta u) - f(x - \delta u) \big), 
\end{align}
that have lower estimation variance and showed that ZO with these two-point feedbacks achieves a convergence rate of $\mathcal{O}(\frac{1}{\sqrt{T}})$ (or $\mathcal{O}(\frac{1}{T})$ when the problem is smooth), which is order-wise much faster than the convergence rate achieved by ZO algorithms with one-point feedback. Therefore, as also pointed out in \cite{larson2019derivative}, a fundamental question we seek to answer in this paper is:

\begin{itemize}[topsep=0pt, noitemsep, leftmargin=*]
	\item {\em (Q1): Does there exist a one-point feedback for which zeroth-order optimization can achieve the same query complexity as that of two-point feedback methods? }
\end{itemize}

\footnotetext[3]{In convex setting, the accuracy is meaured by $f(x) - f(x^\ast) \leq \epsilon$, 
where $x^\ast = \arg \min_{x \in \mathbb{R}^d} f(x)$, 
while in the non-convex setting, it is measured by $\|\nabla f(x)\|^2 \leq \epsilon$ when the objective function is smooth. When the objective function is non-smooth, we enforce two optimality measures, $|f(x) - f_\delta(x)| \leq \epsilon_f$ and $\|\nabla f_\delta(x)\|^2 \leq \epsilon$ together, 
where function $f_\delta(x)$ is a smoothed function defined as $f_\delta(x) := \mathbb{E}_{u\sim \mathcal{N}(0,1)}[f(x+\delta u)]$. 
(UN) means the oracle considers uncontrollable data samples. 
The notations $C^{0,0}$ and $C^{1,1}$ represent the function classes that are either Lipschitz, or have Lipschitz gradient. The detailed definition of these notations can be found in Definition~\ref{def:Lipschitz}.
}

The literature discussed above focuses on deterministic optimization problems (P). Nevertheless, in practice, many problems involve randomness in the environment and parameters, giving rise to the following stochastic optimization problem
\begin{align*}
	\min _{x\in \mathbb{R}^d} f(x) = \mathbb{E}_{\xi} [ F(x, \xi)], \tag{Q}
\end{align*}
where only a noisy function evaluation $F(x, \xi)$ with a random data sample $\xi$ is available. 
ZO algorithms have also been developed to solve the above problem (Q), e.g., \cite{ghadimi2013stochastic,duchi2015optimal,hu2016bandit,bach2016highly,gasnikov2017stochastic,akhavan2020exploiting}. In particular, \cite{ghadimi2013stochastic} consider the following widely-used stochastic two-point feedback
\begin{equation}
	\widetilde{\nabla} f(x) = \frac{u}{\delta} \big(F(x + \delta u, \xi) - F(x, \xi)\big) \label{eq: two_stochastic}
\end{equation}
and show that ZO with this stochastic two-point feedback has the same convergence rate as ZO with the two-point feedback scheme in \eqref{eqn:TwoPoint} for deterministic problems (P). Similarly, \cite{duchi2015optimal} further analyzed the oracle in \eqref{eq: two_stochastic} in a mirror descent framework and showed a similar convergence speed. Stochastic one-point and two-point feedback schemes with improved convergence rates have also been studied in \cite{gasnikov2017stochastic}. However, these stochastic two-point feedback schemes assume that the data sample $\xi$ is controllable, i.e., one can fix the data sample $\xi$ and evaluate the function value at two distinct points $x$ and $x+\delta u$. This assumption is unrealistic in many applications. For example, in reinforcement learning, controlling the sample $\xi$ requires applying the same sequence of noises to the dynamical system and reward function. Hence, two-point feedback schemes with fixed data samples can be impractical. To address this challenge, \cite{hu2016bandit,bach2016highly,akhavan2020exploiting} proposed a more practical noisy two-point feedback method that replaces the fixed sample $\xi$ in \eqref{eq: two_stochastic} with two independent samples $\xi, \xi'$. Its convergence rate was shown to match that of the stochastic one-point feedback $\widetilde{\nabla} f(x) = \frac{u}{\delta}F(x + \delta u, \xi)$. Still though, this two-point feedback method with independent data samples produces gradient estimates with lower variance compared to the conventional one-point feedback method. Therefore, an additional fundamental question we seek to answer in this paper is:
\begin{itemize}[topsep=0pt, noitemsep, leftmargin=*]
	\item {\em (Q2): Can we develop a stochastic one-point feedback that achieves the same practical performance as that of the noisy two-point feedback? }
\end{itemize}

{\bf Contributions:} In this paper, we provide positive answers to these open questions by introducing a new one-point residual feedback scheme and theoretically analyzing the convergence of zeroth-order optimization using this feedback scheme.  Specifically, our contributions are as follows.
We propose a new one-point feedback scheme which requires a single function evaluation at each iteration. This feedback scheme estimates the gradient using the residual between two consecutive feedback points and we refer to it as residual feedback.
We show that our residual feedback induces a smaller estimation variance than the one-point feedback~\eqref{eqn:OnePoint} considered in \cite{flaxman2005online,gasnikov2017stochastic}. Specifically, in deterministic optimization where the objective function is Lipschitz-continuous, we show that ZO with our residual feedback achieves the same convergence rate as existing ZO with two-point feedback schemes. To the best of our knowledge, this is the first one-point feedback scheme with provably comparable performance to two-point feedback schemes in ZO. Moreover, when the objective function has an additional smoothness structure, we further establish an improved convergence rate of ZO with residual feedback. In the stochastic case where only noisy function values are available, we show that the convergence rate of ZO with residual feedback matches the state-of-the-art result of ZO with two-point feedback under uncontrollable data samples. Hence, our residual feedback bridges the theoretical gap between ZO with one-point feedback and ZO with two-point feedback. A summary of the complexity results for the proposed residual-feedback scheme can be found in Table~\ref{table:Complexity}.

{\bf Applications in Learning and Control:} The proposed one-point residual-feedback oracle has important applications in a variety of learning and control problems where the gradients are unavailable or difficult to compute. For example, it can be used to reduce the number of black-box function evaluations, compared to the conventional one-point oracle, in optimal charging problems for electrical vechicles \cite{li2021surrogate}, extreme seeking problems for ABS control for automotive brakes \cite{poveda2021robust,nevsic2009extremum}. In addition, residual feedback can reduce the computational cost of ZO methods for distributed reinforcement learning problems, while maintaining a similar convergence rate as that achieved by two-point methods \cite{zhang2020cooperative}. This is because residual feedback, being a one-point method, requires only a single policy evaluation (generally an expensive calculation) at each iteration to estimate the policy gradient. Moreover, residual feedback can significantly improve the convergence speed of ZO algorithms for non-stationary reinforcement learning problems, as shown in \cite{zhang2020boosting}. Note that two-point methods can not be used for non-stationary reinforcement learning problems because they require two different policy evaluations in the same environment, which is not possible when the environment is non-stationary and changes after each policy evaluation. Compared to these works, here we focus on the iteration complexity of ZO methods with one-point residual feedback for static optimization problems, under different assumptions on the objective functions and their evaluation. This analysis, that is summarized in Table~\ref{table:Complexity}, lays the theoretical foundations of residual feedback and justifies its use for the more challenging learning and control problems discussed above.

\section{Preliminaries}
\label{sec:prelim}
In this section, we present definitions and preliminary results needed throughout our analysis.
Following \cite{nesterov2017random,bach2016highly}, we introduce the following classes of Lipschitz and smooth functions.

\begin{defn}[Lipschitz functions]
	\label{def:Lipschitz}
	The class of Lipschtiz-continuous functions $C^{0,0}$ satisfy: for any $f \in C^{0,0}$, $|f(x) - f(y)| \le L_0 \|x-y\|, ~\forall x,y\in \mathbb{R}^d$,
	for some Lipschitz parameter $L_0>0$. The class of smooth functions $C^{1,1}$ satisfy: for any $f \in C^{1,1}$, $\|\nabla f(x) - \nabla f(y)\| \le L_1\|x-y\|, ~\forall x,y\in \mathbb{R}^d,$
	for some Lipschitz parameter $L_1>0$.
\end{defn}


In ZO, the objective is to estimate the first-order gradient of a function using zeroth-order oracles. Necessarily, we need to perturb the function around the current point along all the directions uniformly in order to estimate the gradient. This motivates us to consider the Gaussian-smoothed version of the function $f$ as introduced in \cite{nesterov2017random}, $f_\delta(x) := \mathbb{E}_{u\sim \mathcal{N}(0,1)}[f(x+\delta u)]$,
where the coordinates of the vector $u$ are i.i.d standard Gaussian random variables. The following bounds on the approximation error of the function $f_\delta(x)$ have been developed in \cite{nesterov2017random}.
\begin{lem}[Gaussian approximation]
	\label{lem:GaussianApprox}
	Consider a function $f$ and its Gaussian-smoothed version $f_\delta$. It holds that 
	\begin{align*}
	& |f_\delta(x)-f(x)|\le
	\begin{cases}
	\delta L_0\sqrt{d}, ~\text{if}~f\in C^{0,0}, \\
	\delta^2 L_1 d, ~\text{if}~f\in C^{1,1},
	\end{cases} \nonumber \\
	& \text{and }
	\|\nabla f_\delta(x)- \nabla f(x)\|\le \delta L_1 (d+3)^{3/2}, ~\text{if}~ f\in C^{1,1}.
	\end{align*}
\end{lem}
Moreover, the smoothed function $f_\delta(x)$ has the following nice geometrical property as proved in \cite{nesterov2017random}. 

\begin{lem}
	\label{lem:SmoothedFunctionLipschitiz}
	If function $f\in C^{0,0}$ is $L_0$-Lipschitz, then its Gaussian-smoothed version $f_\delta$ belongs to $C^{1,1}$ with Lipschitz constant $L_1 = \sqrt{d}\delta^{-1}L_0$.
\end{lem}

We also introduce the following notions of convexity.
\begin{defn}[Convexity]
	A continuously differentiable function $f: \mathbb{R}^d \to \mathbb{R}$ is called convex if for all $x,y \in \mathbb{R}^d$, $	f(x) \ge f(y) + \inner{x-y}{\nabla f(y)}.$
\end{defn}

\section{Deterministic ZO with Residual Feedback}
\label{sec:det}
In this section, we consider the problem (P), 
where the objective function evaluation is fully deterministic. To solve this problem, we propose a zeroth-order estimate of the gradient based on the following {\it one-point residual feedback} scheme
\begin{align}
\widetilde{g}(x_t) := \frac{u_t}{\delta}\big(f(x_t + \delta u_t) - f(x_{t-1} + \delta u_{t-1})\big), \footnotemark[4] \label{eqn:GradientEstimate_Noiseless}
\end{align}
where $u_{t-1}$ and $u_t$ are independent random vectors sampled from the standard multivariate Gaussian distribution. To elaborate, the gradient estimate in \eqref{eqn:GradientEstimate_Noiseless} evaluates the function value at one perturbed point $x_t + \delta u_t$ at each iteration $t$ and the other function value evaluation $f(x_{t-1} + \delta u_{t-1})$ is inherited from the previous iteration. 
Therefore, it is a one-point feedback scheme based on the residual between two consecutive feedback points, and we name it {\it one-point residual feedback}. 
Next, we show that this estimator is an unbiased gradient estimate of the smoothed function $f_\delta(x)$ at $x_t$.
\begin{lem}
	\label{lem:UnbiasedEstimate_Noiseless}
	We have $\mathbb{E}\big[\tilde{g}(x_t)\big] = \nabla f_\delta(x_t)$ for all $x_t \in  \mathbb{R}^d$.
\end{lem}
\begin{pf}
	The proof is straightforward because $u_t$ is independent from $u_{t-1}$ and has zero mean.  \qed
\end{pf}

\footnotetext[4]{At time $t = 0$, we can query the objective function $f$ at $x_0 + \delta u_0$ and update $x_0$ using the conventional one-point oracle~\eqref{eqn:OnePoint}. Then, starting from time $t=1$, we can update using estimator~\eqref{eqn:GradientEstimate_Noiseless}.}

Since $\tilde{g}(x_t)$ is an unbiased estimate of $\nabla f_\delta(x_t)$, we can use it in Stochastic Gradient Descent (SGD) as follows
\begin{equation}
\label{eqn:SGD}
x_{t+1} = x_t - \eta \tilde{g} (x_t),
\end{equation}
where $\eta$ is the stepsize. To analyze the convergence of the above ZO algorithm with residual feedback, we need to bound the variance of the gradient estimate under proper choices of the exploration parameter $\delta$ in \eqref{eqn:GradientEstimate_Noiseless} and the stepsize $\eta$. In the following result, we present the bounds on the second moment of the gradient estimate $\mathbb{E}[\| \tilde{g}(x_t)\|^2]$, which will be used in our analysis later.

\begin{lem}
	\label{lem:BoundSecondMoment_Det}
	Consider a function $f \in C^{0,0}$ with Lipschitz constant $L_0$. Then, under the SGD update rule in~\eqref{eqn:SGD}, the second moment of the residual feedback satisfies
	\begin{equation*}
	\mathbb{E}[\|\tilde{g} (x_t)\|^2] \leq \frac{2 d L_0^2 \eta^2}{\delta^2} \mathbb{E}[ \|\tilde{g}(x_{t-1})\|^2] + 8L_0^2 (d+4)^2.
	\end{equation*}
	Furthermore, if $f(x)$ also belongs to $C^{1,1}$ with constant $L_1$, then the second moment of the residual feedback satisfies
	\begin{align}
	\mathbb{E}[\|& \tilde{g} (x_t)\|^2] \leq  \; \frac{2 d L_0^2 \eta^2}{\delta^2} \mathbb{E}[ \|\tilde{g}(x_{t-1})\|^2] \nonumber \\
	& + 8(d+4)^2 \|\nabla f(x_{t-1})\|^2  + 4L_1^2 (d+6)^3 \delta^2.
	\end{align}
\end{lem}
The proof of above Lemma~\ref{lem:BoundSecondMoment_Det} can be found in Appendix~\ref{sec:BoundSecondMoment_Det}. Lemma~\ref{lem:BoundSecondMoment_Det} shows that the second moment of the residual feedback $\mathbb{E}[\|\tilde{g} (x_t)\|^2]$ can be bounded by a perturbed contraction under the SGD update rule. This perturbation term is crucial to establish the iteration complexity of ZO with our residual feedback. In particular, with the traditional one-point feedback, the perturbation term is in the order of $O({\delta^{-2}})$ and significantly degrades the convergence speed \cite{hu2016bandit}. In comparison, our residual feedback induces a much smaller perturbation term. Specifically, when $f\in C^{0,0}$, the perturbation is the order of $O(L_0^2d^2)$ that is independent of $\delta$, and when $f\in C^{1,1}$, the perturbation is in the order of $O(d^2\|\nabla f(x_{t-1})\|^2 +L_1^2d^3\delta^2)$. 
Therefore, ZO with our residual feedback can achieve a better iteration complexity than that of ZO with the traditional one-point feedback.
\vspace{-3pt}
\subsection{Convergence Analysis}
We first consider the case where the objective function $f$ is nonconvex. When $f$ is differentiable, we say a solution $x$  is $\epsilon$-accurate if $\mathbb{E}[\| \nabla f(x) \|^2 ]\le \epsilon$. However, when $f$ is nonsmooth, the gradient of the original objective function $\nabla f(x)$ does not exist. On the other hand, the smoothed objective function $f_\delta(x)$ is differentiable. Therefore, we find an $\epsilon$-accurate solution of the smoothed problem such that $\mathbb{E}[ \| \nabla f_\delta(x) \|^2] \leq \epsilon$. In the meantime, we require $f_\delta$ to be $\epsilon_f$-close to the original objective function $f$, which requires $\delta \leq \frac{\epsilon_f}{L_0 \sqrt{d}}$ according to Lemma~\ref{lem:GaussianApprox}. Similar optimality conditions have also been considered in \cite{nesterov2017random}.
Under this setup, the convergence rate of ZO with residual feedback is presented below. For simplicity, all the complexity results in this paper are presented in $\mathcal{O}$ notations. The proofs and the explicit form of the constant terms can be found in the supplementary material.

\begin{thm}
	\label{thm:Nonconvex_Noiseless}
	Assume that $f \in C^{0,0}$ with Lipschitz constant $L_0$ and that $f$ is also bounded below by $f^\ast$. Moreover, assume that SGD in~\eqref{eqn:SGD} with residual feedback is run for $T>1/\epsilon_f$ iterations and that $\tilde{x}$ is selected from the $T$ iterates uniformly at random. Let also $\eta = \frac{\sqrt{\epsilon_f}}{2d L_0^2 \sqrt{T}}$ and $\delta = \frac{\epsilon_f}{ L_0 d^{\frac{1}{2}}}$. Then, we have that $\mathbb{E}\big[ \|\nabla f_\delta(\tilde{x}) \|^2\big]  = \mathcal{O}(d^2 \epsilon_f^{-0.5} T^{-0.5})$.
\end{thm}

The proof can be found in Appendix~\ref{sec:NonsmoothNonconvex_Det}. Based on the above convergence rate result, the required iteration complexity to achieve a point $x$ that satisfies $|f(x) - f_\delta(x)| \leq \epsilon_f$ as well as $\mathbb{E}[ \|\nabla f(\tilde{x})\|^2 ] \leq \epsilon$ is of the order $\mathcal{O}(\frac{d^{4}}{ \epsilon_f \epsilon^2})$. This complexity result is close to the complexity result $\mathcal{O}(\frac{d^{3}}{ \epsilon_f \epsilon^2})$ of ZO with two-point feedback in \cite{nesterov2017random}.
When $f(x) \in C^{1,1}$ is a smooth function, we obtain the following convergence rate result for ZO with residual feedback.
\begin{thm}
	\label{thm:Nonconvex_NoiselessSmooth}
	Assume that $f(x) \in C^{0,0}$ with Lipschitz constant $L_0$ and that $f(x) \in C^{1,1}$ with Lipschitz constant $L_1$. Moreover, assume that SGD in \eqref{eqn:SGD} with residual feedback is run for $T$ iterations and that $\tilde{x}$ is selected from the $T$ iterates uniformly at random.
	Let also $\eta = \frac{1} { \widetilde{L} (d+4)^2 T^{\frac{1}{3}} }$, and $\delta = \frac{1}{\sqrt{d} T^{\frac{1}{3}}}$, where $\widetilde{L} = \max(2L_0,$ $32L_1)$. Then, we have that $\mathbb{E}\big[ \|\nabla f(\tilde{x})\|^2\big] = \mathcal{O}(d^2 T^{-\frac{2}{3}})$.
\end{thm} 

The proof can be found in Appendix~\ref{sec:SmoothNonconvex_Det}. In particular, to achieve a point $x$ that satisfies $\mathbb{E}\big[ \|\nabla f(\tilde{x})\|^2\big] \leq \epsilon$, the required iteration complexity is of the order $\mathcal{O}(d^3\epsilon^{-\frac{3}{2}})$. To the best of our knowledge, the best complexity result for ZO with two-point feedback is of the order $\mathcal{O}(d\epsilon^{-1})$, which is established in \cite{nesterov2017random}.
Next, we consider the case where the objective function $f$ is convex. In this case, the optimality of a solution $x$ is measured via the loss gap $f(x) - f(x^\ast)$, where $x^\ast$ is the global optimum of $f$. 
\begin{thm}
	\label{thm:Convex_Noiseless}
	Assume that $f(x)\in C^{0,0}$ is convex with Lipschitz constant $L_0$. Moreover, assume that SGD in \eqref{eqn:SGD} with residual feedback is run for $T$ iterations and define the running average $\bar{x} = \frac{1}{T} \sum_{t=0}^{T-1} x_t$. Let also $\eta = \frac{1}{2 d L_0 \sqrt{T}}$ and $\delta = \frac{1}{\sqrt{T}}$.  Then, we have that $f(\bar{x})  - f(x^\ast) = \mathcal{O}(d T^{-0.5})$.

	Moreover, assume that additionally $f(x) \in C^{1,1}$ with Lipschitz constant   $L_1$, and let $\eta = \frac{1}{2 \tilde{L} (d+4)^2 T^{\frac{1}{3}}}$ and $\delta = \frac{\sqrt{d}}{T^{\frac{1}{3}}}$, where $\tilde{L} = \max\{L_0, 16L_1\}$. Then, we have that $f(\bar{x})  - f(x^\ast) = \mathcal{O}(d^2 T^{-\frac{2}{3}})$.
\end{thm}

The proof can be found in Appendix~\ref{sec:Convex_Det}. To elaborate, to achieve a solution $x$ that satisfies $f(\bar{x}) - f(x^\ast) \leq \epsilon$, the required iteration complexity is of the order $\mathcal{O}(d^2\epsilon^{-2})$ when $f\in C^{0,0}$. Such a complexity result significantly improves the complexity $\mathcal{O}(d^2\epsilon^{-4})$ of ZO with the traditional one-point feedback and is slightly worse than the best complexity $\mathcal{O}(d\epsilon^{-2})$ of ZO with two-point feedback. On the other hand, when $f(x) \in C^{1,1}$, the required iteration complexity of ZO with residual feedback further reduces to $\mathcal{O}(d^{3}\epsilon^{-1.5})$, which is better than the complexity $\mathcal{O}(d\epsilon^{-3})$ of ZO with the traditional one-point feedback whenever $\epsilon< d^{-4/3}$.

\vspace{-5pt}
\section{Online ZO with Stochastic Residual Feedback}
\label{sec:stochastic}
In this section, we study the Problem (Q) where the objective function takes the form $f(x) := \mathbb{E} [ F(x, \xi)]$ and only noisy samples of the function value $F(x, \xi)$ are available. Specifically, we propose the following stochastic residual feedback
\begin{equation}
\label{eqn:GradientEstimate_Noise}
\tilde{g}(x_t) := \frac{u_t}{\delta} \big(F(x_t + \delta u_t, \xi_t) - F(x_{t-1} + \delta u_{t-1}, \xi_{t-1})\big),	
\end{equation}
where $\xi_{t-1}$ and $\xi_t$ are independent random samples that are sampled in iterations $t-1$ and $t$, respectively. We note that our stochastic residual feedback is more practical than most existing two-point feedback schemes, which require the data samples to be controllable, i.e., one can query the function value at two different variables using the same data sample. This assumption is unrealistic in applications where the environment is dynamic. For example, in reinforcement learning \cite{malik2018derivative}, these data samples can correspond to random initial states, noises added to the dynamical system, and reward functions. Therefore, controlling the data samples requires to hard reset the system to the exact same initial state and apply the same sequence of noises, which is impossible when the data is collected from a real-world system. Our stochastic residual feedback scheme in \eqref{eqn:GradientEstimate_Noise} does not suffer from the same issue since it does  not restrict the data sampling procedure. Instead, it simply takes the residual between two consecutive stochastic feedback points. In particular, it is straightforward to show that \eqref{eqn:GradientEstimate_Noise} is an unbiased gradient estimate of the objective function $f_\delta(x)$. Next, we present some assumptions that are used in our analysis later.
\begin{asmp}
	\label{asmp:BoundedVariance}
	(Bounded Variance) We assume that for any $x \in \mathbb{R}^d$ there exists $\sigma>0$ such that
	\vspace{-4pt}
	\begin{equation*}
		\label{eqn:BoundedVarianceFunc}
		\mathbb{E}\big[ \big(F(x, \xi) - f(x) \big)^2 \big] \leq \sigma^2.
	\end{equation*}
\end{asmp}
\vspace{-4pt}
Assumption~\ref{asmp:BoundedVariance} implies that $\mathbb{E}\big[ \big(F(x, \xi_1) -  F(x, \xi_2) \big)^2 \big] \leq 4\sigma^2$. Furthermore, we make the following smoothness assumption in the stochastic setting.
\begin{asmp}
	\label{asmp:BoundedLipschitz}
	Let function $F(x, \xi) \in C^{0,0}$ with Lipschitz constant $L_0(\xi)$. We assume that $L_0(\xi) \leq L_0$ for all $\xi \in \Xi$. In addition, let the function $F(x, \xi) \in C^{1,1}$ with Lipschitz constant $L_1(\xi)$. We assume that $L_1(\xi) \leq L_1$ for all $\xi \in \Xi$.
\end{asmp}

The following lemma provides an upper bound of $\mathbb{E}[ \|\tilde{g}(x_t)\|^2 ]$ in this stochastic setting.

\begin{lem}
	\label{lem:BoundSecondMoment_Stoch}
	Let Assumptions~\ref{asmp:BoundedVariance} and \ref{asmp:BoundedLipschitz} hold and assume $F(x, \xi) \in C^{0,0}$ with Lipschitz constant $L_{0}(\xi)$. We have that 
	\vspace{-4pt}
	\begin{equation*}
	\mathbb{E}[\|\tilde{g} (x_t)\|^2] \leq \frac{4 L_{0}^2 d \eta^2}{\delta^2}\mathbb{E}[\|\tilde{g}(x_{t-1})\|^2] + 16 L_{0}^2 (d + 4)^2 + \frac{8 \sigma^2 d}{\delta^2}.
	\end{equation*}
\end{lem}
The proof can be found in Appendix~\ref{sec:SecondMomentBound_Stoch}. If we assume that $F(x, \xi) \in C^{1,1}$, the upper bound on the above second moment can be further improved (see supplementary material for the details). However, this improvement does not yield a better iteration complexity due to the uncontrollable samples $\xi_t$ and $\xi_{t-1}$. More specifically, the uncontrollable samples lead to an additional term $\frac{8 \sigma^2 d}{\delta^2}$ in the above second moment bound. According to the analysis in \cite{hu2016bandit}, such a term can significantly degrade the iteration complexity. 
\subsection{Convergence Analysis}
Next, we analyze the iteration complexity of ZO with stochastic residual feedback for both non-convex and convex problems.

\begin{thm}
	\label{thm:NoisyOnline_Nonconvex}
	Let Assumptions~\ref{asmp:BoundedVariance} and \ref{asmp:BoundedLipschitz} hold and assume also that $F(x, \xi) \in C^{0,0}$. Moreover, assume that SGD in \eqref{eqn:SGD} with residual feedback is run for $T>1/(d \epsilon_f)$ iterations and that $\tilde{x}$ is selected from the $T$ iterates uniformly at random. Let also
	$\eta = \frac{\epsilon_f^{1.5}}{2\sqrt{2}L_0^2 d^{1.5} \sqrt{T}}$ and $\delta = \frac{\epsilon_f}{L_0 \sqrt{d}}$. Then, we have that $\mathbb{E}\big[ \|\nabla f_\delta(\tilde{x}) \|^2\big] = \mathcal{O}(d^{1.5}\epsilon_f^{-1.5} {T}^{-0.5})$.
	
	Furthermore, assume that additionally $F(x, \xi) \in C^{1,1}$, and that SGD in \eqref{eqn:SGD} with residual feedback is run for $T>2$ iterations. Let also $\eta = \frac{1}{2L_{0} d^{\frac{4}{3}} T^{\frac{2}{3}}}$ and $\delta = \frac{1}{ d^{\frac{5}{6}} T^\frac{1}{6}}$. Then, the output $\tilde{x}$ that is sampled uniformly from the $T$ iterates satisfies $\mathbb{E}\big[ \|\nabla f(\tilde{x})\|^2\big] = \mathcal{O}(d^\frac{4}{3}T^{-\frac{1}{3}})$.
\end{thm}

The proof can be found in Appendix~\ref{sec:Proof_NoisyOnlineNonconvex}. Based on the above results, when $F(x, \xi)$ is non-smooth, to achieve the $\epsilon-$stationary point $\mathbb{E}\big[ \|\nabla f_\delta(\tilde{x}) \|^2\big] \leq \epsilon$ and $|f(x) - f_\delta(x)| \leq \epsilon_f$,  $\mathcal{O}(\frac{d^3}{\epsilon_f^3 \epsilon^2})$ iterations are needed. In addition, if the function $F(x, \xi)$ also satisfies $F(x, \xi) \in C^{1,1}$, then $\mathcal{O}(\frac{d^4}{\epsilon^3})$ iterations are needed to find the $\epsilon-$stationary point of the original function $f(x)$.
Next, we provide the iteration complexity results when the Problem (Q) is convex. 
\begin{thm}
	\label{thm:NoisyOnline_Convex}
	Let Assumptions~\ref{asmp:BoundedVariance} and \ref{asmp:BoundedLipschitz} hold and assume that the function $F(x, \xi) \in C^{0,0}$ is also convex. Moreover, assume that SGD in \eqref{eqn:SGD} with residual feedback is run for $T$ iterations and define the running average $\bar{x} = \frac{1}{T} \sum_{t=0}^{T-1} x_t$. Let also $\eta = \frac{1}{2\sqrt{2} L_0 \sqrt{d} T^{\frac{3}{4}}}$ and $\delta = \frac{1}{T^{\frac{1}{4}}}$. Then, we have that $f(\bar{x})  - f(x^\ast) = \mathcal{O}(\sqrt{d}T^{-\frac{1}{4}})$.
	Moreover, assume that additionally $F(x, \xi) \in C^{1,1}$, and let $\eta = \frac{1}{2\sqrt{2} L_0 d^\frac{2}{3} T^{\frac{2}{3}}}$ and $\delta = \frac{1}{  d^\frac{1}{6} T^{\frac{1}{6}}}$. Then, we have that $f(\bar{x})  - f(x^\ast) = \mathcal{O}(d^\frac{2}{3}T^{-\frac{1}{3}})$.
\end{thm}
The proof can be found in Appendix~\ref{sec:Proof_NoisyOnlineConvex}. According to Theorem~\ref{thm:NoisyOnline_Convex},  $\mathcal{O}(\frac{d^2}{\epsilon^4})$ iterations are needed to achieve $f(\bar{x}) - f(x^\ast) \leq \epsilon$ with a nonsmooth objective function. On the other hand, if $f(x) \in C^{1,1}$, the iteration complexity is improved to $\mathcal{O}(\frac{d^{2}}{\epsilon^{3}})$.

\section{ZO with Mini-batch Stochastic Residual Feedback}
When applying zeroth-order oracles to practical applications, instead of directly using the oracle~\eqref{eqn:GradientEstimate_Noise}, a mini-batch scheme can be implemented to further reduce the variance of the gradient estimate, as discussed in \cite{fazel2018global}. To be more specific, consider the gradient estimate with batch size $b$:
\begin{equation*}
	\tilde{g}_b(x_t) =\frac{u_t}{b \delta} \big(F( x_t + \delta u_t, \xi_{1:b}) - F(x_{t-1} + \delta u_{t-1}, \xi_{1:b}') \big),
\end{equation*}
where $  F( x_t + \delta u_t, \xi_{1:b}) = \sum_{j = 1}^{b}  F( x_t + \delta u_t, \xi_{j})$. It is straightforward to see that the variance of $\tilde{g}_b(x_t)$ is $b^2$ times smaller than that of the oracle~\eqref{eqn:GradientEstimate_Noise}. This is particularly useful when the problem is sensitive to bad search directions. For example, in the policy optimization problem \cite{fazel2018global}, when the gradient has large variance, it can drive the policy parameter to divergence and result in infinite cost. A Mini-batch scheme can reduce the variance of the policy gradient (search direction) estimate and therefore is of particular interest in this scenario. In this paper, we show that using the oracle~\eqref{eqn:GradientEstimate_Noise} in a mini-batch scheme can achieve the same query complexity as standard SGD. Its analysis is provided in Appendix~\ref{sec:Proof_Minibatch}.
\vspace{-2mm}

\section{Numerical Experiments}
\label{sec:exp}
In this section, we demonstrate the effectiveness of the residual one-point feedback scheme for both deterministic and stochastic problems. In the deterministic case, we compare the performance of the proposed oracle with the original one-point feedback and two-point feedback schemes, for the quadratic programming (QP) example considered in \cite{shamir2013complexity}. In the stochastic case, we employ the stochastic variants of above oracles to optimize the policy parameters in a Linear Quadratic Regulation (LQR) problem considered in \cite{fazel2018global,malik2018derivative}. It is shown that the proposed residual one-point feedback significantly outperforms the traditional one-point feedback and its convergence rate matches that of the two-point oracles in both deterministic and stochastic cases. Furthermore, we apply our residual-feedback zeroth-order gradient estimate to solve a large-scale stochastic multi-stage decision making problem to demonstrate its performance in the high dimensional problems.
All experiments are conducted using Matlab R2018b on a 2018 Macbook Pro with a 2.3 GHz Quad-Core Intel Core i5 and 8GB 2133MHz memory. 

In all the experiments, we first manually select the exploration parameter $\delta$. Then, we tune the stepsize $\eta$ so that all algorithms converge at their fastest speed.
\subsection{A Deterministic Scenario: QP Problem}
\begin{figure}[t]
	\centering
	\subfigure[\label{fig:comparison}]{
	\includegraphics[width=0.9\columnwidth]{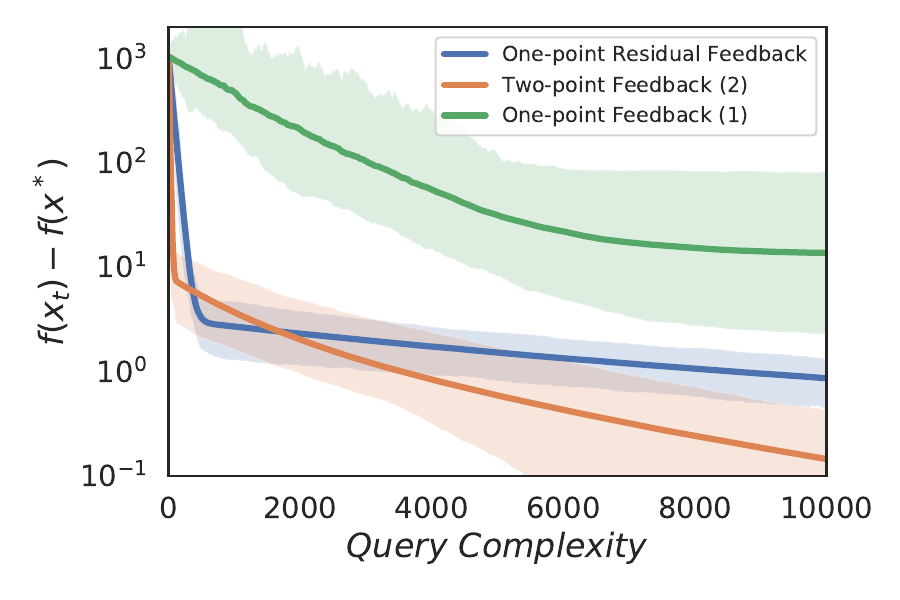}}
	\subfigure[\label{fig:lqr}]{
	\includegraphics[width=0.9\columnwidth]{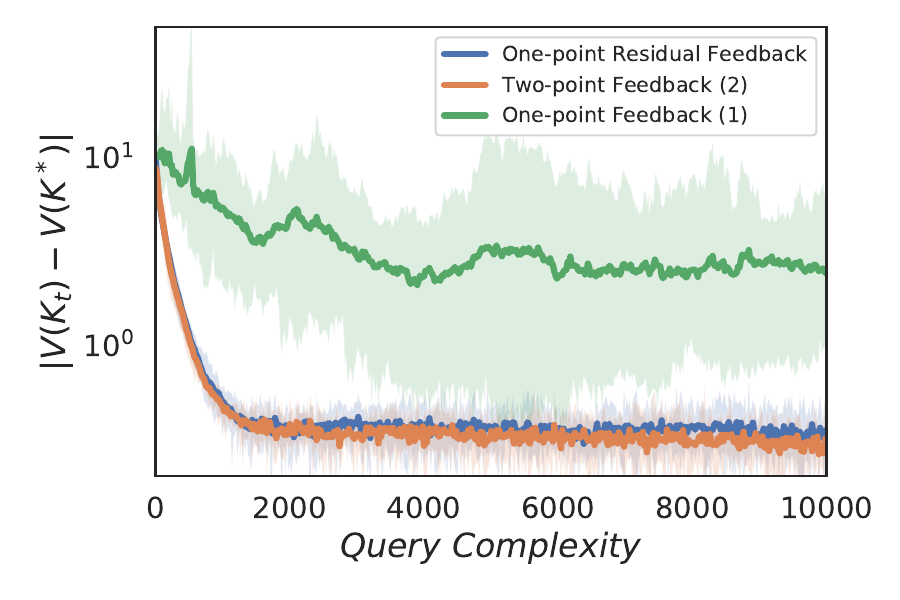}}
	\vspace{-2mm}
	\caption{\small The convergence rate of applying the proposed residual one-point feedback~\eqref{eqn:GradientEstimate_Noiseless} (blue), the two-point oracle \eqref{eqn:TwoPoint} in \cite{nesterov2017random} (orange) and the one-point oracle~\eqref{eqn:OnePoint} in \cite{flaxman2005online} (green) to two problems. In (a), the convergence of $f(x_t) - f(x^\ast)$ in a deterministic QP problem is presented. In (b), the convergence of the costs of policies in the stochastic LQR problem is presented. }
\end{figure}
As in \cite{shamir2013complexity}, consider the QP example $\min \frac{1}{2} (x - c)^T M (x - c)$, 
where $x, c \in \mathbb{R}^{30}$ and $M \in \mathbb{R}^{30 \times 30}$ is a positive semi-definite matrix. This constitutes a convex and smooth problem. The vector $c$ is randomly generated from a uniform distribution in $[0, 2]$. The matrix $M = PP^T$, where each entry in $P \in \mathbb{R}^{30 \times 29}$ is sampled from a uniform distribution in $[0,1]$. The initial point is the origin. For every algorithm, we manually optimize the selection of the exploration parameter $\delta$ and stepsize $\eta$ and run it $100$ times. Specifically, we select $\delta$ as $\delta = 0.1$, and the stepsizes for the proposed residual feedback estimator, the two-point estimator and the conventional one-point estimator are $0.05, 0.1, 0.01$, respectively. The convergence of the function value $f(x) - f(x^\ast)$ is presented in Figure~\ref{fig:comparison}. We observe that the proposed oracle converges as fast as the two-point oracle~\eqref{eqn:TwoPoint} when the iterates are far from the optimizer but achieve less accuracy in the end. Both methods find the optimal function value much faster than the one-point feedback studied in \cite{flaxman2005online,gasnikov2017stochastic}. These observations validate our theoretical results in Section~\ref{sec:det}.
\subsection{A Stochastic Scenario: Policy Optimization}
We use the proposed residual feedback to optimize the policy parameters in a LQR problem, as in \cite{fazel2018global,malik2018derivative}. Specifically, consider a system whose state $x_k \in \mathbb{R}^{n_x}$ at time $k$ is subject to the dynamical equation $x_{k+1} = A x_k + B u_k + w_k$,
where $u_k \in \mathbb{R}^{n_u}$ is the control input at time $k$, $A \in \mathbb{R}^{n_x \times n_x}$ and $B \in \mathbb{R}^{n_x \times n_u}$ are dynamical matrices that are unknown, and $w_k$ is the noise on the state transition. Moreover, consider a state feedback policy $u_k = K x_k$, where $K \in \mathbb{R}^{n_u \times n_x}$ is the policy parameter. Policy optimization essentially aims to find the optimal policy parameter $K$ so that the discounted accumulated cost function $V(K) := \mathbb{E}\big[ \sum_{t = 0}^{\infty} \gamma^{t} (x_k^T Q x_k + u_k^T R u_k)\big]$
is minimized, where $\gamma \leq 1$ is the discount factor. 

In our simulation, we select $n_x = 6$, $n_u = 6$ and $\gamma = 0.5$. Therefore, the problem has dimension $d = 36$.
When implementing the policy $u_k = K_t x_k$, due to the noise $w_k$, evaluation of the cost of the policy $K_t$ is noisy.
We apply the one-point feedback~\eqref{eqn:OnePoint} with noise \cite{gasnikov2017stochastic}, two-point feedback with uncontrolled noise \cite{hu2016bandit,bach2016highly} and the residual one-point feedback~\eqref{eqn:GradientEstimate_Noise} to solve the above policy optimization problem. 
To evaluate the cost $V(K_t)$ given the policy parameter $K_t$ at iteration $t$, we run one episode with a finite horizon length $H = 50$. 
The dynamical matrices $A$ and $B$ are randomly generated and the noise $w_k$ is sampled from a Gaussian distribution $\mathcal{N}(0, 0.1^2)$. We select the exploration parameter $\delta$ as $\delta = 0.1$, and the stepsizes for the proposed residual feedback estimator, the two-point estimator and the conventional one-point estimator are $2\times10^{-3}, 2.5\times10^{-3}, 1.5\times10^{-4}$, respectively. 
We run each algorithm $10$ times. At each trial, all the algorithms start from the same initial guess of the policy parameter  $K_0$, which is generated by perturbing the optimal policy parameter $K^\ast$ with a random matrix, as in \cite{malik2018derivative}. Each entry in this random perturbation matrix is sampled from a uniform distribution in $[0, 0.2]$. The performance of all the algorithms over $10$ trials is measured in terms of $|V(K_t) - V(K^\ast)|$ and is presented in Figure~\ref{fig:lqr}. 
We observe that the residual one-point feedback~\eqref{eqn:GradientEstimate_Noise} converges much faster than the one-point oracle in \cite{gasnikov2017stochastic} and has comparable query complexity to the two-point feedback under uncontrolled noises considered in \cite{hu2016bandit,bach2016highly}. This corroborates our theoretical analysis in Section~\ref{sec:stochastic}.

\subsection{Zeroth-Order Policy Optimization for a Large-Scale Multi-Stage Decision Making Problem}
\label{sec:LargeScale}
In this section, we consider a large-scale multi-stage resource allocation problem. Specifically, we consider $16$ agents that are located on a $4 \times 4$ grid. At agent $i$, resources are stored in the amount of $m_i(k)$ and there is also a demand for resources in the amount of $d_i(k)$ at instant $k$.
In the meantime, agent $i$ also decides what fraction of resources $a_{ij}(k) \in [0, 1]$ it sends to its neighbors $j \in  \mathcal{N}_i$ on the grid. The local amount of resources and demands at agent $i$ evolve as $m_i(k+1) = m_i(k) - \sum_{j \in \mathcal{N}_i} a_{ij}(k) m_i(k) + \sum_{j \in \mathcal{N}_i} a_{ji}(k) m_j(k) - d_i(k)$ and $d_i(k) = A_i \sin(\omega_i k + \phi_i) + w_{i,k}$,
where the amplitude $A_i$ is sampled uniformly from $[1,2]$, $\omega_i = 2 \pi$, $\phi_i$ is uniformly sampled from $[0, \pi]$, and  $w_{i,k}$ is the noise in the demand sampled from the normal distribution $\mathcal{N}(0, A_i^2/100)$. At time $k$, agent $i$ receives a local reward $r_i(k)$, such that $r_i(k) = 0$ when $m_i(k) \geq 0$ and $r_i(k) = -m_i(k)^2$ when $m_i(k) < 0$. Let agent $i$ makes its decisions according to a parameterized policy function $\pi_{i, \theta_i}(o_i): \mathcal{O}_i \rightarrow [0,1]^{|\mathcal{N}_i|}$, where $\theta_i$ is the parameter of the policy function $\pi_i$, $o_i \in \mathcal{O}_i$ denotes agent $i$'s observation, and $|\mathcal{N}_i|$ represents the number of agent $i$'s neighbors on the grid.

Our goal is to train a policy that can be executed in a fully distributed way based on agents' local information. Specifically, during the execution of policy functions $\{\pi_{i, \theta_i}(o_i)\}$, we let each agent only observe its local amount of resource $m_i(k)$ and demand and $d_i(k)$, i.e., $o_i(k) = [m_i(k), d_i(k)]^T$.
In addition, the policy function $\pi_{i, \theta_i}(o_i)$ is parameterized as the following: $a_{ij} = \exp(z_{ij}) / \sum_{j}\exp(z_{ij})$, where $z_{ij} = \sum_{p = 1}^{9} \psi_p(o_i) \theta_{ij}(p)$ and $\theta_i = [\dots, \theta_{ij}, \dots]^T$. Specifically, the feature function $\psi_p(o_i)$ is selected as $\psi_p(o_i) = \|o_i - c_p\|^2$, where $c_p$ is the parameter of the $p$-th feature function. Specifically, $c_p$ are selected as vectors lying in the two-dimensional grid $(-0.5, 0, 0.5)^2$. The goal for the agents is to find an optimal policy $\pi^\ast = \{\pi_{i, \theta_i}(o_i)\}$ so that the global accumulated reward
\begin{align}
	\label{eqn:Obj}
	J(\theta) = \sum_{i=1}^{16} \sum_{k = 0}^{K} \gamma^{k} r_i(k) 
\end{align} 
\begin{figure}[t]
	\centering
	\subfigure[\label{fig:resource}]{
		\includegraphics[width=0.9\columnwidth]{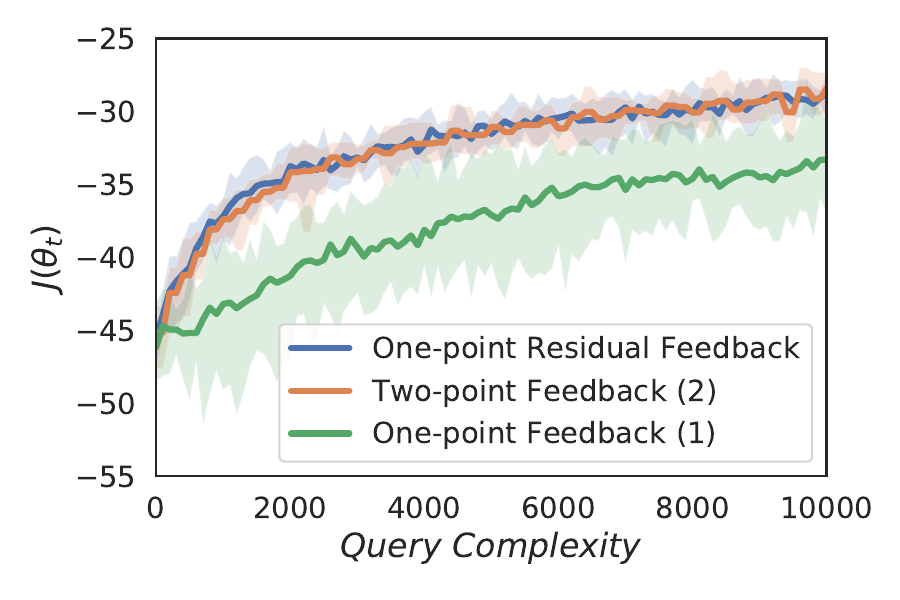}}
	\subfigure[\label{fig:navigation}]{
		\includegraphics[width=0.9\columnwidth]{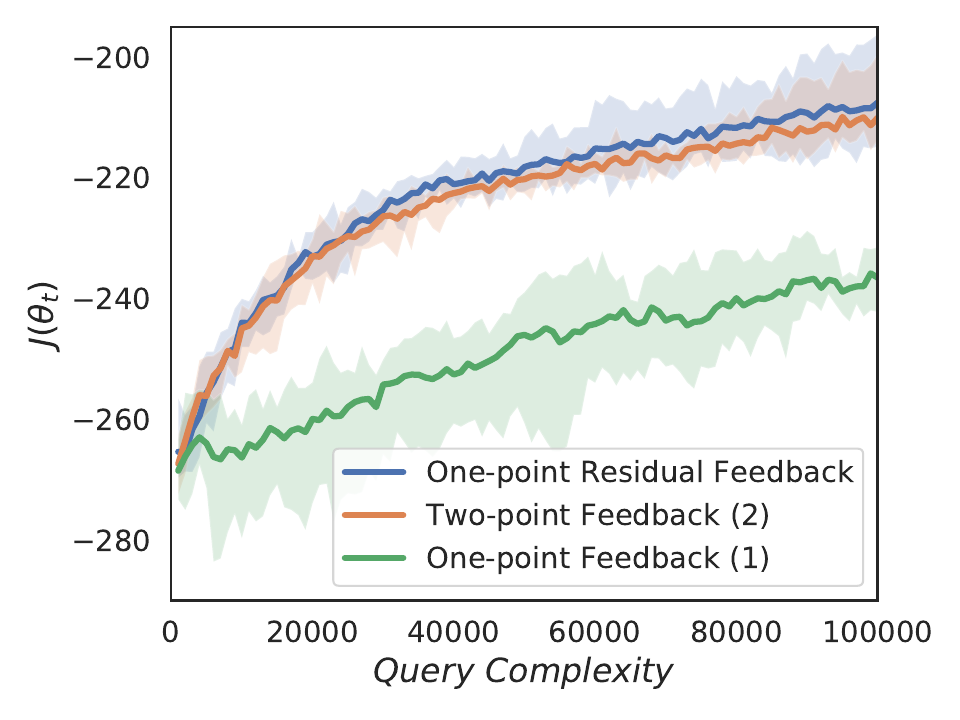}}
	\caption{\small The convergence rate of applying the proposed residual one-point feedback~\eqref{eqn:GradientEstimate_Noiseless} (blue), the two-point oracle \eqref{eqn:TwoPoint} in \cite{nesterov2017random} (orange) and the one-point oracle~\eqref{eqn:OnePoint} in \cite{flaxman2005online} (green) to the large-scale stochastic multi-stage resource allocation problem and the multi-robot cooperative navigation problem. The vertical axis represents the total rewards and the horizontal axis represents the number of episodes the agents take to evaluate their policy parameter iterates during the policy optimization procedure. }
\end{figure}
is maximized, where $\theta = [\dots, \theta_i, \dots]$ is the global policy parameter, $K$ is the horizon of the problem, and $\gamma$ is the discount factor. Effectively, the agents need to make decisions on $64$ actions, and each action is decided by $9$ parameters. Therefore, the problem dimension is $d = 576$.
To implement zeroth-order policy gradient estimators \eqref{eqn:OnePoint} and \eqref{eqn:GradientEstimate_Noise} to find the optimal policy, at iteration $t$, we let all agents implement the policy with parameter $\theta_t + \delta u_t$, collect rewards $\{r_i(k)\}$ at  time instants $k = 0, 1, \dots, K$ and compute the noisy policy value according to~\eqref{eqn:Obj}. Then, the zeroth-order policy gradient is estimated using~\eqref{eqn:OnePoint} or \eqref{eqn:GradientEstimate_Noise}. On the contrary, when the two-point zeroth-order policy gradient estimator~\eqref{eqn:TwoPoint} is used, at each iteration $k$, all agents need to evaluate two policies $\theta_t \pm \delta u_t$ to update the policy parameter once. In Figure~\ref{fig:resource}, we present the performance of using zeroth-order policy gradients~\eqref{eqn:OnePoint}, \eqref{eqn:TwoPoint} and \eqref{eqn:GradientEstimate_Noise} to solve this large-scale multi-stage resource allocation problem, where the discount factor is set as $\gamma = 0.75$ and the length of horizon $K = 30$. 
We select the exploration parameter $\delta$ as $\delta = 0.1$, and the stepsizes for the proposed residual feedback estimator, the two-point estimator and the conventional one-point estimator are $1\times10^{-4}, 1\times10^{-4}, 5\times10^{-5}$, respectively. Each algorithm is run for $10$ trials. 
We observe that policy optimization with the proposed residual-feedback gradient estimate~\eqref{eqn:GradientEstimate_Noise} improves the optimal policy parameters with the same learning rate as the two-point zeroth-order gradient estimator~\eqref{eqn:TwoPoint}, where the learning rate is measured by the number of episodes the agents take to evaluate the policy parameter iterates.
In the meantime, both estimators perform much better than the one-point policy gradient estimate~\eqref{eqn:OnePoint} considered in \cite{fazel2018global,malik2018derivative}.

\subsection{Zeroth-Order Policy Optimization for a Multi-Robot Cooperative Navigation Problem}
In this section, we demonstrate the effectiveness of the proposed one-point residual-feedback gradient estimator using the benchmark multi-agent particle environment \cite{lowe2017multi}. Specifically, we consider the two-agent two-landmark cooperative nagivation task, where the agents navigate to the landmarks in the environment without colliding into each other. At each time step, agent $i$ observes a vector $o_i \in \mathbb{R}^{12}$ consisting of all agents' states, i.e., their positions and velocities, and the two landmarks' positions. Then, agent $i$ selects a $5$ dimensional action vector based on its observation $o_i$, to drive itself around the world . The dynamics of the agents' states are governed by the physical engine used in the particle environment. At each time, the team of agents receive a team reward $r(k) = - \sum_{l = 1, 2} \min_{i = 1, 2} \|pos_i - pos_l\|$, where $l$ denotes the landmark index,  $pos_i$ and $pos_l$ represent the position vectors of agent $i$ and landmark $l$. In addition, if the agents collide at time step $k$, the team receives $-1$ as a penalty. In each episode, there are 25 time steps. 

We let each agent learn a policy function $\pi_{i,\theta_i}(o_i*)$ that is designed as a ReLU neural network with one hidden layer, where $\theta_i$ denotes the weights. The hidden layer consists of $32$ neurons. Therefore, each neural network policy function has $(12+1)\times32 + (32+1)\times5 = 581$ parameters to learn. The dimension of the problem is $d = 1162$. Since a ReLU activation function is used, the policy optimization problem is non-smooth. We implement the proposed residual-feedback policy gradient estimator, as well as the conventional one-point estimator~\eqref{eqn:OnePoint} and the two-point estimator~\eqref{eqn:TwoPoint}, for $5$ trials. Specifically, we select the exploration parameter $\delta$ as $\delta = 0.1$, and the stepsizes for the proposed residual feedback estimator, the two-point estimator and the conventional one-point estimator are $5\times10^{-6}, 1\times10^{-5}, 1\times10^{-6}$, respectively. The learning rates for these algorithms are manually tuned to achieve their best performance respectively. The results are presented in Figure~\ref{fig:navigation}. In this non-smooth setting, we observe that policy optimization with the proposed residual-feedback gradient still has comparable performance to that of the two-point policy gradient estimator~\eqref{eqn:TwoPoint} and both estimators perform much better than the one-point policy gradient estimate~\eqref{eqn:OnePoint}, similar to the smooth case in Section 6.3. 

\section{Conclusion}
In this paper, we proposed a residual one-point feedback oracle for zeroth-order optimization, which estimates the gradient of the objective function using a single query of the function value at each iteration. When the function evaluation is noiseless, we showed that ZO using the proposed oracle can achieve the same iteration complexity as ZO using two-point oracles when the function is non-smooth. When the function is smooth, this complexity of ZO can be further improved. This is the first time that a one-point zeroth-order oracle is shown to match the performance of two-point oracles in ZO. In addition, we considered a more realistic scenario where the function evaluation is corrupted by noise. We showed that the convergence rate of ZO using the proposed oracle matches the best known results using one-point feedback or two-point feedback with uncontrollable data samples. We provided numerical experiments that showed that the proposed oracle outperforms the one-point oracle and is as effective as two-point feedback methods.

\bibliographystyle{model5-names}        
\bibliography{biblio}           



\appendix

\noindent {\textbf{Appendix}}

\section{Proof of Lemma~\ref{lem:BoundSecondMoment_Det}}
\label{sec:BoundSecondMoment_Det}
First, we show the bound when $f(x) \in C^{0,0}$. Recalling the expression of $\tilde{g}(x_t)$ in \eqref{eqn:GradientEstimate_Noiseless}, we have that
\begin{align}
& \mathbb{E}[\|\tilde{g} (x_t)\|^2] = \mathbb{E}[\frac{1}{\delta^2} \big(f(x_t + \delta u_t) - f(x_{t-1} + \delta u_{t-1})\big)^2 \|u_t\|^2] & \nonumber \\
& \leq \frac{2}{\delta^2} \mathbb{E}[ \big(f(x_t + \delta u_t) - f(x_{t-1} + \delta u_{t})\big)^2 \|u_t\|^2] & \nonumber  \\
& + \frac{2}{\delta^2} \mathbb{E}[ \big( f(x_{t-1} + \delta u_{t}) - f(x_{t-1} + \delta u_{t-1})\big)^2 \|u_t\|^2]. & \nonumber
\end{align}
Since function $f \in C^{0,0}$ with Lipschitz constant $L_0$, we obtain that
\begin{align}
\label{eqn:BSM_2}
\mathbb{E}[\|\tilde{g} (x_t)\|^2] \leq & \frac{2 L_0^2}{\delta^2} \mathbb{E}[ \|x_t - x_{t-1}\|^2 \|u_t\|^2 ] & \nonumber \\
& + 2L_0^2 \mathbb{E}[\|u_t - u_{t-1}\|^2\|u_t\|^2]. & 
\end{align}
Since $u_t$ is independently sampled from $x_t - x_{t-1}$, we have that $\mathbb{E}[ \|x_t - x_{t-1}\|^2 \|u_t\|^2 ] = \mathbb{E}[\|x_t - x_{t-1}\|^2] \mathbb{E}[\|u_t\|^2]$. Since $u_t$ is subject to standard multivariate normal distribution, $\mathbb{E}[\|u_t\|^2] = d$. Furthermore, using Lemma 1 in \cite{nesterov2017random}, we get that $\mathbb{E}[\|u_t - u_{t-1}\|^2\|u_t\|^2] \leq 2\mathbb{E}[(\|u_t\|^2 + \|u_{t-1}\|^2)\|u_t\|^2] = 2\mathbb{E}[(\|u_t\|^4] +2\mathbb{E}[ \|u_{t-1}\|^2\|u_t\|^2] \leq 4(d+4)^2$. Plugging these bounds into inequality~\eqref{eqn:BSM_2}, we have that
\begin{equation*}
\label{eqn:BSM_3}
\mathbb{E}[\|\tilde{g} (x_t)\|^2] \leq \frac{2 d L_0^2}{\delta^2} \mathbb{E}[ \|x_t - x_{t-1}\|^2] + 8L_0^2 (d+4)^2.
\end{equation*}
Since $x_t = x_{t-1} - \eta \tilde{g}(x_{t-1})$, we get that
\begin{equation*}
\label{eqn:BSM_4}
\mathbb{E}[\|\tilde{g} (x_t)\|^2] \leq \frac{2 d L_0^2 \eta^2}{\delta^2} \mathbb{E}[ \|\tilde{g}(x_{t-1})\|^2] + 8L_0^2 (d+4)^2.
\end{equation*}
Next, we show the bound when we have the additional smoothness condition $f(x) \in C^{1,1}$ with constant $L_1$. Given the gradient estimate in \eqref{eqn:GradientEstimate_Noiseless}, we have that
\begin{equation}
\label{eqn:BSM_5}
\mathbb{E}[ \| \tilde{g}(x_t) \|^2 ] \leq \mathbb{E}[ \frac{ (f(x_t + \delta u_t) - f(x_{t-1} + \delta u_{t-1}) )^2 }{\delta^2} \|u_t\|^2 ].
\end{equation}
Next, we bound the term $( f(x_t + \delta u_t) - f(x_{t-1} + \delta u_{t-1}) )^2$. Adding and subtracting $f(x_{t-1} + \delta u_{t})$ inside the square, and applying the inequality $(a + b)^2 \leq 2a^2 + 2b^2$, we can obtain
\begin{align}
\label{eqn:BSM_6}
& ( f(x_t + \delta u_t) - f(x_{t-1} + \delta u_{t-1}) )^2  & \nonumber \\
& \leq 2( f(x_t + \delta u_t) - f(x_{t-1} + \delta u_{t}) )^2  & \nonumber \\
& \quad + 2( f(x_{t-1} + \delta u_{t}) - f(x_{t-1} + \delta u_{t-1}) )^2. &
\end{align} 
Since the function $f(x)$ is also Lipschitz continuous with constant $L_0$, we get that
\begin{align}
\label{eqn:BSM_7}
( f(x_t + \delta u_{t}) - f(x_{t-1} + & \delta u_{t}) )^2   \leq L_0^2 \|x_t - x_{t-1}\|^2 & \nonumber \\
&  = L_0^2 \eta^2 \|\tilde{g}(x_{t-1})\|^2. &
\end{align}
Next, we bound the term $( f(x_{t-1} + \delta u_t) - f(x_{t-1} + \delta u_{t-1}) )^2$. Adding and subtracting $f(x_{t-1})$, $\langle \nabla f(x_{t-1}), \delta u_t \rangle$ and $\langle \nabla f(x_{t-1}), \delta u_{t-1} \rangle$ inside the square term, we have that
\begin{align}
\label{eqn:BSM_8}
& ( f(x_{t-1} + \delta u_t) - f(x_{t-1} + \delta u_{t-1}) )^2 & \nonumber \\
& \leq  2\langle \nabla f(x_{t-1}), \delta (u_t - u_{t-1}) \rangle^2  &  \\
&  + 4(f(x_{t-1} + \delta u_t) - f(x_{t-1}) - \langle \nabla f(x_{t-1}),   \delta u_t \rangle)^2 & \nonumber \\
& + 4(f(x_{t-1} + \delta u_{t-1}) - f(x_{t-1}) - \langle \nabla f(x_{t-1}), \delta u_{t-1} \rangle)^2.  & \nonumber
\end{align}
Since $f(x) \in C^{1,1}$ with constant $L_1$, we get that $| f(x_{t-1} + \delta u_t) - f(x_{t-1}) - \langle \nabla f(x_{t-1}), \delta u_t \rangle | \leq \frac{1}{2} L_1 \delta^2 \|u_t\|^2$, according to (6) in \cite{nesterov2017random}. And similarly, we also have $|f(x_{t-1} + \delta u_{t-1}) - f(x_{t-1}) - \langle \nabla f(x_{t-1}), \delta u_{t-1} \rangle| \leq \frac{1}{2} L_1 \delta^2 \|u_{t-1}\|^2$. Substituting these inequalities into \eqref{eqn:BSM_8}, we obtain that
\begin{align}
\label{eqn:BSM_9}
& ( f(x_{t-1} + \delta u_t) - f(x_{t-1} + \delta u_{t-1}) )^2 \leq 2\langle \nabla f(x_{t-1}), & \nonumber \\
& \;\; \delta (u_t - u_{t-1}) \rangle^2  + L_1^2 \delta^4 \|u_t\|^4 + L_1^2 \delta^4 \|u_{t-1}\|^4. &
\end{align}
Moreover, substituting the inequalities~\eqref{eqn:BSM_7} and \eqref{eqn:BSM_9} in the upper bound in \eqref{eqn:BSM_6}, we get that
\begin{align}
\label{eqn:BSM_10}
& ( f(x_t + \delta u_t) - f(x_{t-1} + \delta u_{t-1}) )^2  & \nonumber \\
\leq & 2 L_0^2 \eta^2 \|\tilde{g}(x_{t-1})\|^2 +  4\langle \nabla f(x_{t-1}), \delta (u_t - u_{t-1}) \rangle^2  & \nonumber \\
& + 2L_1^2 \delta^4 \|u_t\|^4 + 2L_1^2 \delta^4 \|u_{t-1}\|^4 &
\end{align}
Using the bound \eqref{eqn:BSM_10} in inequality~\eqref{eqn:BSM_5}, and applying the bounds $\mathbb{E}[ \|u_t\|^6 ] \leq (d+6)^3$ and $\mathbb{E}[\|u_{t-1}\|^4 \|u_t\|^2] \leq (d+6)^3$, we have that
\begin{align}
\label{eqn:BSM_11}
& \mathbb{E}[ \| \tilde{g}(x_t) \|^2 ] \leq \frac{2d L_0^2 \eta^2}{\delta^2}\mathbb{E}[ \| \tilde{g}(x_{t-1}) \|^2 ] &  \\
& + 4 \mathbb{E}[ \langle \nabla f(x_{t-1}), u_t - u_{t-1} \rangle^2 \|u_t\|^2 ] + 4L_1^2 (d+6)^3 \delta^2. & \nonumber
\end{align}
Since $\langle \nabla f(x_{t-1}), u_t - u_{t-1} \rangle^2 \leq 2 \langle \nabla f(x_{t-1}), u_t \rangle^2 + 2\langle \nabla f(x_{t-1}), u_{t-1} \rangle^2$, we get that
\begin{align}
\label{eqn:BSM_12}
& \mathbb{E}[ \langle \nabla f(x_{t-1}), u_t - u_{t-1} \rangle^2 \|u_t\|^2 ] \leq 2 \mathbb{E}[ \langle \nabla f(x_{t-1}),   & \nonumber \\
& \quad u_t \rangle^2 \|u_t\|^2 ]+ 2 \mathbb{E}[ \langle \nabla f(x_{t-1}), u_{t-1} \rangle^2 \|u_t\|^2 ].
\end{align}
For the term $\mathbb{E}[ \langle \nabla f(x_{t-1}), u_{t-1} \rangle^2 \|u_t\|^2 ]$, we have that $\mathbb{E}[ \langle \nabla f(x_{t-1}), u_{t-1} \rangle^2 \|u_t\|^2 ] \leq \mathbb{E}[ \|\nabla f(x_{t-1})\|^2 \|u_{t-1}\|^2$ $\|u_t\|^2 ] \leq d^2 \mathbb{E}[ \|\nabla f(x_{t-1})\|^2 ]$. For the term $\mathbb{E}[ \langle \nabla f(x_{t-1}),$  $u_t \rangle^2 \|u_t\|^2 ]$, according to Theorem 3 in \cite{nesterov2017random}, we have a stronger bound $\mathbb{E}[ \langle \nabla f(x_{t-1}),$ $u_t \rangle^2 \|u_t\|^2 ] \leq (d + 4)  \mathbb{E}[ \|\nabla f(x_{t-1})\|^2 ]$. Substituting these bounds into \eqref{eqn:BSM_12}, and because $d^2 + d + 4 \leq (d + 4)^2$, we have that
\begin{align}
\label{eqn:BSM_13}
& \mathbb{E}[ \langle \nabla f(x_{t-1}), u_t - u_{t-1} \rangle^2 \|u_t\|^2 ] & \nonumber \\
& \quad \quad \quad \quad \quad  \leq 2(d+4)^2 \mathbb{E}[ \|\nabla f(x_{t-1})\|^2 ].
\end{align}
Substituting the bound \eqref{eqn:BSM_13} into inequality~\eqref{eqn:BSM_11}, we complete the proof.

\section{Proof of Theorem~\ref{thm:Nonconvex_Noiseless}}
\label{sec:NonsmoothNonconvex_Det}
Since we have that $f(x) \in C^{0,0}$, according to Lemma~\ref{lem:GaussianApprox}, the function $f_\delta(x)$ has $L_1(f_\delta)$-Lipschitz continuous gradient where $L_1(f_\delta) = \frac{\sqrt{d}}{\delta} L_0$. Furthermore, according to Lemma 1.2.3 in \cite{nesterov2013introductory}, we can get the following inequality
\begin{align}
\label{eqn:Inequality_1}
& f_\delta(x_{t+1}) \leq f_\delta(x_{t}) +  \langle \nabla f_\delta(x_t), x_{t+1} - x_t \rangle & \nonumber \\
& \quad \quad \quad \quad \quad   + \frac{L_1(f_\delta)}{2} \|x_{t+1} - x_t\|^2 & \nonumber \\
& = f_\delta(x_{t}) - \eta \langle \nabla f_\delta(x_t), \tilde{g}(x_t) \rangle + \frac{L_1(f_\delta) \eta^2}{2} \|\tilde{g} (x_t)\|^2  & \nonumber \\
& = f_\delta(x_{t}) - \eta \langle \nabla f_\delta(x_t), \Delta_t \rangle - \eta \|\nabla f_\delta(x_t)\|^2 & \nonumber \\
& \quad \quad \quad \quad \quad \quad  \quad \quad \quad \quad  \quad + \frac{L_1(f_\delta) \eta^2}{2} \|\tilde{g}(x_t)\|^2,
\end{align}
where $\Delta_t = \tilde{g}(x_t) - \nabla f_\delta(x_t)$. According to Lemma~\ref{lem:UnbiasedEstimate_Noiseless}, we can get that $\mathbb{E}_{u_t} [ \tilde{g}(x_t) ] = \nabla f_\delta(x_t)$. Therefore, taking expectation over $u_t$ on both sides of inequality~\eqref{eqn:Inequality_1} and rearranging terms, we have that
\begin{align}
\label{eqn:Inequality_2}
\eta \mathbb{E}[\|\nabla f_\delta(x_t)\|^2 ] \leq & \mathbb{E}[f_\delta(x_t)] - \mathbb{E}[f_\delta(x_{t+1})] & \nonumber \\
&  + \frac{L_1(f_\delta) \eta^2}{2} \mathbb{E}[ \|\tilde{g}(x_t)\|^2].
\end{align}
Telescoping above inequalities from $t = 0$ to $T-1$ and dividing both sides by $\eta$, we obtain that
\begin{align}
\label{eqn:Inequality_3}
\sum_{t = 0}^{T-1} \mathbb{E}[\|\nabla f_\delta(x_t)\|^2 ] \leq & \frac{\mathbb{E}[f_\delta(x_0)] - \mathbb{E}[f_\delta(x_{T})]}{\eta} & \nonumber \\
& + \frac{L_1(f_\delta) \eta}{2} \sum_{t = 0}^{T-1} \mathbb{E}[\|\tilde{g}(x_t)\|^2] & \nonumber \\
\leq \frac{\mathbb{E}[f_\delta(x_0)] - f_\delta^\ast}{\eta} + & \frac{L_1(f_\delta) \eta}{2} \sum_{t = 0}^{T-1} \mathbb{E}[\|\tilde{g}(x_t)\|^2],
\end{align}
where $f_\delta^\ast$ is the lower bound of the smoothed function $f_\delta(x)$. $f_\delta^\ast$ must exist because we assume the orignal function $f(x)$ is lower bounded and the smoothed function has a bounded distance from $f(x)$ due to Lemma~\ref{lem:GaussianApprox}.

Recall the contraction result of the second moment $\mathbb{E}[ \|\tilde{g}(x_t)\|^2] $ in Lemma~\ref{lem:BoundSecondMoment_Det} when $f(x) \in C^{0,0}$. Denote the contraction rate $\frac{2 d L_0^2 \eta^2}{\delta^2}$ as $\alpha$ and the constant perturbation term $M = 8L_0^2 (d+4)^2$. Then, we get that
\begin{equation}
\mathbb{E}[\|\tilde{g} (x_t)\|^2] \leq \alpha^t \mathbb{E}[\|\tilde{g}(x_0)\|^2] + \frac{1 - \alpha^t}{1 - \alpha} M.
\end{equation}
Summing the above inequality over time, we obtain
\begin{align}
\label{eqn:Inequality_6}
& \sum_{t = 0}^{T-1} \|\tilde{g}(x_t)\|^2 \leq \frac{1 - \alpha^{T}}{1 - \alpha} \mathbb{E}[ \|\tilde{g}(x_{0})\|^2] + \sum_{t = 0}^{T-1} \big( \frac{1 - \alpha^t}{1 - \alpha} M \big) \nonumber \\
& \leq  \frac{1}{1 - \alpha} \mathbb{E}[ \|\tilde{g}(x_{0})\|^2] + \frac{1}{1 - \alpha} M T.
\end{align}
Plugging the bound in \eqref{eqn:Inequality_6} into inequality~\eqref{eqn:Inequality_3}, and since $L_1(f_\delta) = \frac{\sqrt{d}}{\delta} L_0$, we have that
\begin{align}
\label{eqn:Inequality_7}
& \sum_{t = 0}^{T-1} \mathbb{E}[\|\nabla f_\delta(x_t)\|^2 ] \leq \frac{\mathbb{E}[f_\delta(x_0)] - f_\delta^\ast}{\eta} + \frac{d^\frac{1}{2} L_0 \eta}{2\delta} & \nonumber \\
& \quad \quad \big(\frac{1}{1 - \alpha} \mathbb{E}[ \|\tilde{g}(x_{0})\|^2] + \frac{1}{1 - \alpha} 8 L_0^2 (d+4)^2 T\big) .
\end{align}
To fullfill the requirement that $|f(x) - f_\delta(x)| \leq \epsilon_f$, we set the exporation parameter $\delta = \frac{\epsilon_f}{d^\frac{1}{2} L_0}$. In addition, let the stepsize be $\eta = \frac{\sqrt{\epsilon_f} }{  2d L_0^2 T^{\frac{1}{2}}}$. 
We have that $\alpha = \frac{1}{2T \epsilon_f} \leq \frac{1}{2}$ and $\frac{1}{1 - \alpha} \leq 2$, when $T \geq \frac{1}{\epsilon_f}$. Plugging the choices of $\eta$ and $\delta$ into inequality~\eqref{eqn:Inequality_7}, we obtain that
\begin{align}
& \sum_{t = 0}^{T-1} \mathbb{E}[\|\nabla f_\delta(x_t)\|^2 ] \leq 2 L_0^2 \big( \mathbb{E}[f_\delta(x_t)] - f_\delta^\ast \big) \frac{d}{\sqrt{\epsilon_f}}\sqrt{T} & \nonumber \\
& \quad \quad \quad \quad \quad \quad + \frac{1}{2\sqrt{\epsilon_f T}}\mathbb{E}[ \|\tilde{g}(x_{0})\|^2]  + 4L_0^2 \frac{(d+4)^2}{\sqrt{\epsilon_f}} \sqrt{T}. & \nonumber 
\end{align}
Dividing both sides of above inequality by $T$, we complete the proof.

\section{Proof of Theorem~\ref{thm:Nonconvex_NoiselessSmooth}}
\label{sec:SmoothNonconvex_Det}
Following the same process in the beginning of the proof of Theorem~\ref{thm:Nonconvex_Noiseless}, we can get
\begin{equation}
\label{eqn:Pf_3.4_1}
\sum_{t = 0}^{T-1} \mathbb{E}[\|\nabla f_\delta(x_t)\|^2 ] \leq \frac{\mathbb{E}[f_\delta(x_0)] - f_\delta^\ast}{\eta} + \frac{L_1\eta}{2} \sum_{t = 0}^{T-1} \mathbb{E}[ \|\tilde{g}(x_t)\|^2 ].
\end{equation}
Since $\frac{1}{2} \mathbb{E}[ \|\nabla f(x_t)\|^2 ] \leq \mathbb{E}[ \|\nabla f_\delta(x_t)\|^2 ] + \mathbb{E}[ \|\nabla f(x_t) - \nabla f_\delta(x_t) \|^2 ]$, and according to the bound~\eqref{eqn:Pf_3.4_1} and Lemma~\ref{lem:GaussianApprox}, we have that
\begin{align}
\label{eqn:Pf_3.4_2}
	& \frac{1}{2} \sum_{t=0}^{T-1} \|\mathbb{E}[ \|\nabla f(x_t)\|^2 ] \leq \frac{\mathbb{E}[f_\delta(x_0)] - f_\delta^\ast}{\eta} & \nonumber \\
	& \quad \quad \quad  + \frac{L_1\eta}{2} \sum_{t = 0}^{T-1} \mathbb{E}[\|\tilde{g}(x_t)\|^2] + L_1^2 (d+3)^3 \delta^2 T.
\end{align}
In addition, similar to the process to derive the bound in \eqref{eqn:Inequality_6}, according to Lemma~\ref{lem:BoundSecondMoment_Det}, when $f(x) \in C^{1,1}$, we can get that
\begin{align}
\label{eqn:Pf_3.4_3}
	& \sum_{t = 0}^{T-1} \|\tilde{g}(x_t)\|^2  \leq \frac{1}{1 - \alpha}  \mathbb{E}[ \|\tilde{g}(x_0)\|^2 ] + \frac{8}{1 - \alpha} (d+4)^2  & \nonumber \\
	& \quad \quad \quad   \sum_{t = 0}^{T-1} \|\nabla f(x_t)\|^2 + \frac{4}{1 - \alpha} L_1^2 (d+6)^3 \delta^2 T.
\end{align}
Plugging the bound~\eqref{eqn:Pf_3.4_3} into \eqref{eqn:Pf_3.4_2}, we have that 
\begin{align}
\label{eqn:Pf_3.5_4}
& \frac{1}{2} \sum_{t=0}^{T-1} \|\mathbb{E}[ \|\nabla f(x_t)\|^2 ] \leq \frac{\mathbb{E}[f_\delta(x_0)] - f_\delta^\ast}{\eta} & \nonumber \\
& \quad \quad \quad + \frac{L_1\eta}{2} \big( \frac{1}{1 - \alpha}  \mathbb{E}[ \|\tilde{g}(x_0)\|^2 ] + \frac{4}{1 - \alpha} L_1^2 (d+6)^3 \delta^2 T \nonumber \\ 
& \quad \quad \quad + \frac{8}{1 - \alpha} (d+4)^2 \sum_{t = 0}^{T-1} \mathbb{E}[\|\nabla f(x_t)\|^2] \big) & \nonumber \\
& \quad \quad \quad + L_1^2 (d+3)^3 \delta^2 T.
\end{align}
Recalling that $\tilde{L} = \max\{32 L_1, 2L_0\}$, let $\eta = \frac{1}{\tilde{L} (d+4)^2 T^\frac{1}{3}}$ and $\delta = \frac{1}{\sqrt{d} T^\frac{1}{3}}$, and we have that $\alpha = 2d L_0^2 \frac{\eta^2}{\delta^2} \leq \frac{1}{2}$. In addition, the coefficient before the term $ \|\nabla f(x_t)\|^2$ in the upper bound above $\frac{L_1 \eta}{2} \frac{8}{1 - \alpha} (d+4)^2 \leq \frac{1}{4}$. Therefore, we obtain that
\begin{align}
& \frac{1}{4} \sum_{t=0}^{T-1} \|\mathbb{E}[ \|\nabla f(x_t)\|^2 ] \leq \tilde{L} (\mathbb{E}[f_\delta(x_0)] - f_\delta^\ast)(d+4)^2 T^\frac{1}{3}  & \nonumber \\
& \quad \quad \quad + \frac{1}{32 (d+4)^2 T^\frac{1}{3}} \mathbb{E}[ \|\tilde{g}(x_0)\|^2 ] + \frac{L_1^2}{8} \frac{(d+6)^3 }{(d+4)^2 d} & \nonumber \\
& \quad \quad \quad + L_1^2 \frac{(d+3)^3}{d} T^\frac{1}{3}. \nonumber
\end{align}
Dividing both sides of above inequality by $T$, we complete the proof.

\section{Proof of Theorem~\ref{thm:Convex_Noiseless}}
\label{sec:Convex_Det}

First, according to iteration~\eqref{eqn:SGD}, we have that
\begin{align}
	& \|x_{t+1} - x^*\|^2 \leq \|x_t - \eta \widetilde{g}(x_t) -x^*\|^2 & \nonumber \\
	& = \|x_{t} - x^*\|^2 -2\eta \inner{\widetilde{g}(x_t)}{x_t - x^*} + \eta^2 \|\widetilde{g}(x_t)\|^2. \nonumber
\end{align}
Taking expectation on both sides, and since $\mathbb{E} [\tilde{g}(x_t)] = \nabla f_\delta (x_t)$, we obtain that
\begin{align}
\label{eqn:Pf_3.5_1}
\mathbb{E}[ \|x_{t+1} - x^*\|^2 ] \leq & \; \mathbb{E}[ \|x_{t} - x^\ast\|^2 ] -2\eta \langle \nabla f_\delta(x_t) , x_t - x^\ast \rangle & \nonumber \\
&  + \eta^2 \mathbb{E}[\|\tilde{g}(x_t)\|^2]. & 
\end{align}
Due to the convexity, we have that $\langle \nabla f_\delta(x_t) , x_t - x^\ast \rangle \geq f_\delta(x_t) - f_\delta(x^\ast)$. Plugging this inequality into \eqref{eqn:Pf_3.5_1}, we have that
\begin{align}
\label{eqn:Pf_3.5_2}
\mathbb{E}[\|x_{t+1} - x^*\|^2] \leq & \; \mathbb{E}[\|x_{t} - x^\ast\|^2] -2\eta (f_\delta(x_t) - f_\delta(x^\ast)) & \nonumber \\
& + \eta^2 \mathbb{E}[ \|\tilde{g}(x_t)\|^2 ]. & 
\end{align}
When $f(x) \in C^{0,0}$, using Lemma~\eqref{lem:GaussianApprox}, we can replace $f_\delta(x)$ with $f(x)$ in above inequality and get
\begin{align}
\mathbb{E}[ \|x_{t+1} - x^*\|^2 ] \leq & \; \mathbb{E}[ \|x_{t} - x^\ast\|^2 ] -2\eta (f(x_t) - f(x^\ast)) & \nonumber \\
& + \eta^2 \mathbb{E}[ \|\tilde{g}(x_t)\|^2 ] + 4  L_0 \sqrt{d} \delta \eta. & \nonumber
\end{align}
Rearranging the terms and telescoping from $t = 0$ to $T-1$, we obtain that
\begin{align}
	& \sum_{t = 0}^{T-1} \mathbb{E}[f(x_t)] - T f(x^\ast) \leq \frac{1}{2\eta} (\|x_{0} - x^\ast\|^2 - \mathbb{E}[ \|x_{T} - x^\ast\|^2 ]) & \nonumber \\
	& \quad \quad \quad \quad \quad \quad \quad \quad \quad \quad \quad + \frac{\eta}{2} \sum_{t = 0}^{T-1} \mathbb{E}[ \|\tilde{g}(x_t)\|^2 ] + 2 L_0 \sqrt{d} \delta T & \nonumber \\
	& \leq \frac{1}{2\eta}  \|x_{0} - x^\ast\|^2  + \frac{\eta}{2} \sum_{t = 0}^{T-1} \mathbb{E}[ \|\tilde{g}(x_t)\|^2 ] + 2 L_0 \sqrt{d} \delta T & \nonumber
\end{align}
Since function $f(x) \in C^{0,0}$, we can plug the bound~\eqref{eqn:Inequality_6} into the above inequality and get that
\begin{align}
	& \sum_{t = 0}^{T-1} \mathbb{E}[ f(x_t)] - T f(x^\ast) \leq \frac{1}{2\eta}  \|x_{0} - x^\ast\|^2  & \nonumber \\
	& + \frac{\eta}{2 (1-\alpha)} \mathbb{E}[\|\tilde{g}(x_0)\|^2] + \frac{4 \eta}{1-\alpha} L_0^2 (d+4)^2 T + 2 L_0 \sqrt{d} \delta T. & \nonumber
\end{align}
Let $\eta = \frac{1}{2 d L_0 \sqrt{T}}$ and $\delta = \frac{1}{\sqrt{T}}$. We have that $\alpha = 2dL_0^2 \frac{\eta^2}{\delta^2} = \frac{1}{2d} \leq \frac{1}{2}$. Therefore, $\frac{1}{1 - \alpha} \leq 2$. Applying this bound and the choice of $\eta$ and $\delta$ into above inequality, we have that
\begin{align}
	& \sum_{t = 0}^{T-1} \mathbb{E}[f(x_t)] - T f(x^\ast) \leq L_0 \|x_{0} - x^\ast\|^2 d\sqrt{T}  & \nonumber \\
	& + \frac{1}{2 d L_0 \sqrt{T}} \mathbb{E}[\|\tilde{g}(x_0)\|^2] + 4 L_0 \frac{(d + 4)^2}{d} \sqrt{T} + 2L_0 \sqrt{d} \sqrt{T}. & \nonumber 
\end{align}
Recalling that $f(\bar{x}) \leq \frac{1}{T}\sum_{t = 0}^{T-1} f(x_t)$ due to  convexity and dividing both sides of above inequality by $T$, the proof of the nonsmooth case is complete.

When function $f(x) \in C^{1,1}$, it is straightforward to see that we also have the inequality~\eqref{eqn:Pf_3.5_2}. In addition, according to Lemma~\ref{lem:GaussianApprox}, we can replace $f_\delta(x)$ with $f(x)$ in above inequality and get
\begin{align}
\label{eqn:Pf_3.5_5}
& \mathbb{E}[ \|x_{t+1} - x^*\|^2 ] \leq \mathbb{E}[\|x_{t} - x^\ast\|^2] -2\eta (f(x_t) - f(x^\ast)) & \nonumber \\
& \quad \quad \quad \quad \quad \quad \quad \quad + \eta^2 \mathbb{E}[\|\tilde{g}(x_t)\|^2] + 4  L_1 d \delta^2 \eta. 
\end{align}
Similarly to the above analysis, we telescope the above inequality from $t = 0$ to $T-1$, apply the bound on $\sum_{t = 0}^{T-1} \mathbb{E}[\|\tilde{g}(x_t)\|^2]$ in \eqref{eqn:Pf_3.4_3} and obtain that
\begin{align}
	& \sum_{t = 0}^{T-1} \mathbb{E}[f(x_t)] - T f(x^\ast) \leq \; \frac{1}{2\eta}  \|x_{0} - x^\ast\|^2 & \nonumber \\
	& \quad \quad \quad + \frac{\eta}{2 (1-\alpha)} \mathbb{E}[\|\tilde{g}(x_0)\|^2]  + \frac{2 \eta}{1-\alpha} L_1^2 (d+6)^3 \delta^2 T & \nonumber \\
	&  \quad \quad \quad + \frac{4 \eta}{1-\alpha} (d+4)^2 \sum_{t = 0}^{T-1} \mathbb{E}[\| \nabla f(x_t)\|^2] + 2 L_1 d \delta^2 T. & \nonumber 
\end{align}
Since $f(x) \in C^{1,1}$ is convex, we have that $\| \nabla f(x_t)\|^2 \leq 2L_1 (f(x_t) - f(x^\ast))$ according to (2.1.7) in \cite{nesterov2013introductory}. Applying this bound into the above inequality, we get that
\begin{align}
& \sum_{t = 0}^{T-1} \mathbb{E}[f(x_t)] - T f(x^\ast) \leq \frac{1}{2\eta}  \|x_{0} - x^\ast\|^2  & \nonumber \\
&  + \frac{\eta}{2 (1-\alpha)} \mathbb{E}[\|\tilde{g}(x_0)\|^2]  + \frac{2 \eta}{1-\alpha} L_1^2 (d+6)^3 \delta^2 T  & \nonumber \\
& + \frac{8\eta}{1-\alpha} L_1 (d+4)^2 \big( \sum_{t = 0}^{T-1} \mathbb{E}[f(x_t)] - T f(x^\ast) \big) + 2 L_1 d \delta^2 T. & \nonumber
\end{align}
Let $\eta = \frac{1}{2 \tilde{L} (d+4)^2 T^{\frac{1}{3}}}$ and $\delta = \frac{\sqrt{d}}{T^{\frac{1}{3}}}$ where $\tilde{L} = \max\{L_0, 16L_1\}$. Then, we have that $\alpha = 2dL_0^2 \frac{\eta^2}{\delta^2} \leq \frac{1}{2(d+4)^4} \leq \frac{1}{2}$. In addition, we have that $\frac{8\eta}{1-\alpha} L_1 (d+4)^2 \leq \frac{1}{2 T^\frac{1}{3}} \leq \frac{1}{2}$. Applying these two bounds into above inequality and rearranging terms, we have that
\begin{align}
& \frac{1}{2} \sum_{t = 0}^{T-1} \mathbb{E}[f(x_t)] - T f(x^\ast) \leq  \tilde{L} \|x_{0} - x^\ast\|^2 (d+4)^2 T^\frac{1}{3} & \nonumber \\
& +  \frac{1}{2 \tilde{L} (d+4)^2 T^{\frac{1}{3}}} \mathbb{E}[\|\tilde{g}(x_0)\|^2]  + \frac{L_1}{8} \frac{(d+6)^3d}{(d+4)^2} + 2L_1 d^2 T^\frac{1}{3}. & \nonumber
\end{align}
Recalling that $f(\bar{x}) \leq \frac{1}{T}\sum_{t = 0}^{T-1} f(x_t)$ due to  convexity and dividing both sides of above inequality by $T$, the proof of the smooth case is complete.

\section{Proof of Lemma~\ref{lem:BoundSecondMoment_Stoch}}
\label{sec:SecondMomentBound_Stoch}
The analysis is similar to the proof in Section~\ref{sec:BoundSecondMoment_Det}. First, consider the case when $F(x, \xi) \in C^{0,0}$ with $L_0(\xi)$. According to \eqref{eqn:GradientEstimate_Noise}, we have that
\begin{align}
& \mathbb{E}[\|\tilde{g} (x_t)\|^2] & \nonumber \\
&  = \mathbb{E}[\frac{1}{\delta^2} \big(F(x_t + \delta u_t, \xi_t) - F(x_{t-1} + \delta u_{t-1}, \xi_{t-1})\big)^2 \|u_t\|^2]  & \nonumber \\
& \leq \frac{2}{\delta^2} \mathbb{E}[ \big(F(x_t + \delta u_t, \xi_t) - F(x_{t-1} + \delta u_{t-1}, \xi_t)\big)^2 \|u_t\|^2] \nonumber \\
& + \frac{2}{\delta^2} \mathbb{E}[ \big( F(x_{t-1} + \delta u_{t-1}, \xi_t) - F(x_{t-1} + \delta u_{t-1}, \xi_{t-1})\big)^2 \|u_t\|^2]. \nonumber
\end{align}
Using the bound in Assumption~\ref{asmp:BoundedVariance}, we get that $\frac{2}{\delta^2} \mathbb{E}[ \big( F(x_{t-1} + \delta u_{t-1}, \xi_t) - F(x_{t-1} + \delta u_{t-1}, \xi_{t-1})\big)^2 \|u_t\|^2] \leq \frac{8 d \sigma^2}{\delta^2}$. In addition, adding and subtracting $F(x_{t-1} + \delta u_{t}, \xi_{t})$ in $\big(F(x_t + \delta u_t, \xi_t) - F(x_{t-1} + \delta u_{t-1}, \xi_t)\big)^2$ in above inequality, we obtain that
\begin{align}
& \mathbb{E}[\|\tilde{g} (x_t)\|^2]\leq  \frac{8 d \sigma^2}{\delta^2} + & \nonumber \\
& \frac{4}{\delta^2} \mathbb{E}[ \big(F(x_t + \delta u_t, \xi_t) - F(x_{t-1} + \delta u_{t}, \xi_t)\big)^2 \|u_t\|^2] & \nonumber \\
& + \frac{4}{\delta^2} \mathbb{E}[ \big( F(x_{t-1} + \delta u_{t}, \xi_t) - F(x_{t-1} + \delta u_{t-1}, \xi_{t})\big)^2 \|u_t\|^2] & \nonumber 
\end{align}
Using Assumption~\ref{asmp:BoundedLipschitz}, we can bound the last two items on the right hand side of above inequality following the same procedure after inequality~\eqref{eqn:BSM_2} and get that 
\begin{equation*}
	\mathbb{E}[\|\tilde{g} (x_t)\|^2] \leq  \frac{4 d L_0^2 \eta^2}{\delta^2} \mathbb{E}[ \|\tilde{g}(x_{t-1})\|^2] + 16L_0^2 (d+4)^2 + \frac{8 d \sigma^2}{\delta^2}.
\end{equation*}
The proof is complete.

\section{Proof of Theorem~\ref{thm:NoisyOnline_Nonconvex}}
\label{sec:Proof_NoisyOnlineNonconvex}
When function $F(x) \in C^{0,0}$ with $L_0(\xi)$, using Assumption~\ref{asmp:BoundedLipschitz} and following the same procedure in Section~\ref{sec:NonsmoothNonconvex_Det}, we have that
\begin{align}
\label{eqn:Pf_4.4_1}
\sum_{t = 0}^{T-1} \mathbb{E}[\|\nabla f_\delta(x_t)\|^2 ]  & \leq \frac{\mathbb{E}[f_\delta(x_0)] - f_\delta^\ast}{\eta} & \nonumber \\
& + \frac{L_1(f_\delta) \eta}{2} \sum_{t = 0}^{T-1} \mathbb{E}[\|\tilde{g}(x_t)\|^2], 
\end{align}
where $L_1(f_\delta) = \frac{\sqrt{d}}{\delta} L_0$. 
In addition, according to Lemma~\ref{lem:BoundSecondMoment_Stoch}, we get that
\begin{align}
	\label{eqn:Pf_4.4_2}
	\sum_{t = 0}^{T-1} \mathbb{E}[  \|\tilde{g}(x_t)\|^2 ] \leq \frac{1}{1 - \alpha} \mathbb{E}[ & \|\tilde{g}(x_{0})\|^2] + \frac{16 L_0^2}{1 - \alpha} (d+4)^2 T & \nonumber \\
	& + \frac{8 \sigma^2}{1 - \alpha}\frac{d}{\delta^2} T, 
\end{align}
where $\alpha =  \frac{4 d L_0^2 \eta^2}{\delta^2}$.
Plugging \eqref{eqn:Pf_4.4_2} into the bound in \eqref{eqn:Pf_4.4_1}, we obtain that
\begin{align}
\label{eqn:Pf_4.4_3}
& \sum_{t = 0}^{T-1} \mathbb{E}[\|\nabla f_\delta(x_t)\|^2 ]  \leq  \frac{\mathbb{E}[f_\delta(x_0)] - f_\delta^\ast}{\eta} + \frac{4 \sigma^2 L_0}{1 - \alpha} d^{1.5} \frac{\eta}{\delta^3} T & \nonumber \\
& + \frac{\sqrt{d} L_0}{2(1 - \alpha)} \mathbb{E}[ \|\tilde{g}(x_{0})\|^2] \frac{\eta}{\delta}  + \frac{8 L_0^3 \sqrt{d}}{1 - \alpha} (d+4)^2 \frac{\eta}{\delta} T . &
\end{align}
Similar to Section~\ref{sec:NonsmoothNonconvex_Det}, to fullfill the requirement that $|f(x) - f_\delta(x)| \leq \epsilon_f$, we set the exporation parameter $\delta = \frac{\epsilon_f}{d^\frac{1}{2} L_0}$. In addition, let the stepsize be $\eta = \frac{\epsilon_f^{1.5} }{  2\sqrt{2} L_0^2 d^{1.5} T^{\frac{1}{2}}}$. 
Then, we have that $\alpha =  \frac{4 d L_0^2 \eta^2}{\delta^2} = \frac{\epsilon_f}{2 d T}\leq \frac{1}{2}$ when $T \geq \frac{1}{d \epsilon_f}$. Therefore, we have that $\frac{1}{1 - \alpha} \leq 2$. Applying this bound and the choices of $\eta$ and $\delta$ into the bound~\eqref{eqn:Pf_4.4_3}, we obtain that 
\begin{align}
& \sum_{t = 0}^{T-1} \mathbb{E}[\|\nabla f_\delta(x_t)\|^2 ]  \leq  2\sqrt{2} L_0^2 (\mathbb{E} [ f_\delta(x_0) ] - f_\delta^\ast ) \frac{d^{1.5} \sqrt{T}}{\epsilon_f^{1.5}} & \nonumber \\
& + \frac{L_0 \epsilon_f^{0.5} }{2 \sqrt{2dT}} \mathbb{E}[ \|\tilde{g}(x_{0})\|^2]  + 4 \sqrt{2} L_0^2 \frac{(d+4)^2}{\sqrt{d}} \sqrt{\epsilon_f T} & \nonumber \\
& +  2\sqrt{2} \sigma^2 L_0^2 \frac{d^{1.5} \sqrt{T}}{\epsilon_f^{1.5}}. & \nonumber 
\end{align}
Dividing both sides by $T$, the proof for the nonsmooth case is complete.

When function $F(x, \xi) \in C^{1,1}$ with $L_1(\xi)$, according to Assumption~\ref{asmp:BoundedLipschitz}, we also have that $f_\delta(x), f(x) \in C^{1,1}$ with constant $L_1$.  Similarly to the proof in Section~\ref{sec:SmoothNonconvex_Det},  we get that 
\begin{align}
\label{eqn:Pf_4.4_5}
& \frac{1}{2} \sum_{t=0}^{T-1} \|\mathbb{E}[ \|\nabla f(x)\|^2 \|] \leq \frac{\mathbb{E}[f_\delta(x_0)] - f_\delta^\ast}{\eta} & \nonumber \\
& + \frac{L_1\eta}{2} \sum_{t = 0}^{T-1} \mathbb{E}[\|\tilde{g}(x_t)\|^2] + L_1^2 (d+3)^3 \delta^2 T.
\end{align}
Plugging inequality~\eqref{eqn:Pf_4.4_2} into the above upper bound, we obtain that
\begin{align}
	& \frac{1}{2} \sum_{t=0}^{T-1} \|\mathbb{E}[ \|\nabla f(x)\|^2 \|] \leq  \frac{\mathbb{E}[f_\delta(x_0)] - f_\delta^\ast}{\eta} & \nonumber \\
	& \quad \quad \quad + \frac{L_1\eta}{2(1 - \alpha)} \mathbb{E}[ \|\tilde{g}(x_{0})\|^2] + \frac{8 L_0^2 L_1 }{1 - \alpha} (d+4)^2 \eta T  & \nonumber \\
	& \quad \quad \quad + \frac{4 L_1 \sigma^2}{1 - \alpha}\frac{d\eta}{\delta^2} T + L_1^2 (d+3)^3 \delta^2 T.
\end{align}
Let $\eta = \frac{1}{2 \sqrt{2} L_{0} d^{\frac{4}{3}} T^{\frac{2}{3}}}$ and $\delta = \frac{1}{ d^{\frac{5}{6}} T^\frac{1}{6}}$. Then, $\alpha = \frac{4 d L_0^2 \eta^2}{\delta^2} = \frac{1}{2 T} \leq \frac{1}{2}$ and $\frac{1}{1 - \alpha} \leq 2$. Plugging these results into the above inequality, we get that
\begin{align}
& \frac{1}{2} \sum_{t=0}^{T-1} \|\mathbb{E}[ \|\nabla f(x)\|^2 \|] \leq  2\sqrt{2} L_0 (\mathbb{E}[f_\delta(x_0)] - f_\delta^\ast) d^{\frac{4}{3}} T^\frac{2}{3} & \nonumber \\
& + \frac{L_1}{2 \sqrt{2} L_{0} d^{\frac{4}{3}} T^{\frac{2}{3}}} \mathbb{E}[ \|\tilde{g}(x_{0})\|^2]   + 4\sqrt{2} L_0 L_1 \frac{(d+4)^2}{d^\frac{4}{3}} T^\frac{1}{3} & \nonumber \\ 
&+ \frac{2\sqrt{2} L_1 \sigma^2}{L_0 d^\frac{1}{3}} T^\frac{1}{3} + L_1^2 \frac{(d+3)^3}{d^\frac{5}{3}}  T^\frac{2}{3}.
\end{align}
Dividing both sides by $T$, the proof for the smooth case is complete.

\section{Proof of Theorem~\ref{thm:NoisyOnline_Convex}}
\label{sec:Proof_NoisyOnlineConvex}

When the function $f(x) \in C^{0,0}$ with constant $L_0(\xi)$ is convex, we can follow the same procedure as in Section~\ref{sec:Convex_Det} and get that
\begin{align}
	\sum_{t = 0}^{T-1} \mathbb{E}[f(x_t)] - T f(x^\ast) & \leq \frac{1}{2\eta}  \|x_{0} - x^\ast\|^2  & \nonumber \\
	& + \frac{\eta}{2} \sum_{t = 0}^{T-1} \mathbb{E}[\|\tilde{g}(x_t)\|^2] + 2 L_0 \sqrt{d} \delta T. & \nonumber
\end{align}
Plugging the bound \eqref{eqn:Pf_4.4_2} into above inequality, we have that
\begin{align}
& \sum_{t = 0}^{T-1} \mathbb{E}[f(x_t)] - T f(x^\ast) \leq  \frac{1}{2\eta}  \|x_{0} - x^\ast\|^2  & \nonumber \\
& \quad \quad \quad \quad + \frac{\eta}{2(1 - \alpha)} \mathbb{E}[ \|\tilde{g}(x_{0})\|^2] + \frac{8 L_0^2}{1 - \alpha} (d+4)^2 \eta T & \nonumber \\
& \quad \quad \quad \quad + \frac{4 \sigma^2}{1 - \alpha}\frac{d\eta}{\delta^2} T + 2 L_0 \sqrt{d} \delta T. &
\end{align}
Let $\eta = \frac{1}{2\sqrt{2} L_0 \sqrt{d} T^{\frac{3}{4}}}$ and $\delta = \frac{1}{T^{\frac{1}{4}}}$. Then, we have that $\alpha = \frac{4 d L_0^2 \eta^2}{\delta^2} = \frac{1}{2T} \leq \frac{1}{2}$. Plugging these results into the above inequality, we get that
\begin{align}
& \sum_{t = 0}^{T-1} \mathbb{E}[f(x_t)] - T f(x^\ast) \leq \sqrt{2} L_0 \|x_{0} - x^\ast\|^2 \sqrt{d} T^\frac{3}{4} & \nonumber \\
& + \frac{1}{2\sqrt{2} L_0 \sqrt{d} T^{\frac{3}{4}}} \mathbb{E}[ \|\tilde{g}(x_{0})\|^2] + 4\sqrt{2} L_0 \frac{(d+4)^2}{\sqrt{d}} T^\frac{1}{4} & \nonumber \\
&  + \frac{2\sqrt{2} \sigma^2}{L_0} \sqrt{d} T^\frac{3}{4} + 2 L_0 \sqrt{d} T^\frac{3}{4}.
\end{align}
Dividing both sides by $T$, the proof for the nonsmooth case is complete.

When the function $f(x) \in C^{1,1}$ with constant $L_1(\xi)$, we can also get the inequality~\eqref{eqn:Pf_3.5_5} in Section~\ref{sec:Convex_Det}. Telescoping this inequality from $t = 0$ to $T-1$ and rearranging terms, we obtain
\begin{align}
& \sum_{t = 0}^{T-1} \mathbb{E}[f(x_t)] - T f(x^\ast) \leq \frac{1}{2\eta}  \|x_{0} - x^\ast\|^2  & \nonumber \\
& \quad \quad \quad \quad  \quad \quad + \frac{\eta}{2} \sum_{t = 0}^{T-1} \mathbb{E}[\|\tilde{g}(x_t)\|^2] + 2 L_1 d \delta^2 T.
\end{align}
Plugging the bound \eqref{eqn:Pf_4.4_2} into above inequality, we have that
\begin{align}
& \sum_{t = 0}^{T-1} \mathbb{E}[f(x_t)] - T f(x^\ast) \leq \frac{1}{2\eta}  \|x_{0} - x^\ast\|^2  + 2 L_1 d \delta^2 T  & \nonumber \\
& + \frac{\eta}{2(1 - \alpha)} \mathbb{E}[ \|\tilde{g}(x_{0})\|^2] + \frac{8 L_0^2}{1 - \alpha} (d+4)^2 \eta T + \frac{4 \sigma^2}{1 - \alpha}\frac{d\eta}{\delta^2} T . & \nonumber
\end{align}
Let $\eta = \frac{1}{2\sqrt{2} L_0 d^\frac{2}{3} T^{\frac{2}{3}}}$ and $\delta = \frac{1}{  d^\frac{1}{6} T^{\frac{1}{6}}}$. Then, we have that $\alpha = \frac{4 d L_0^2 \eta^2}{\delta^2} = \frac{1}{2T} \leq \frac{1}{2}$. Plugging these parameters into above inequality, we get that
\begin{align}
& \sum_{t = 0}^{T-1} \mathbb{E}[f(x_t)] - T f(x^\ast) \leq \sqrt{2} L_0 \|x_{0} - x^\ast\|^2 d^\frac{2}{3} T^\frac{2}{3} & \nonumber \\
&  + \frac{1}{2\sqrt{2} L_0 d^\frac{2}{3} T^{\frac{2}{3}}} \mathbb{E}[ \|\tilde{g}(x_{0})\|^2] + 4\sqrt{2} L_0 \frac{(d+4)^2}{d^\frac{2}{3}}T^\frac{1}{3} & \nonumber \\
& + \frac{2\sqrt{2} \sigma^2}{L_0} d^\frac{2}{3} T^\frac{2}{3} + 2 L_1  d^\frac{2}{3} T^\frac{2}{3}. \nonumber 
\end{align}
Dividing both sides by $T$, the proof for the smooth case is complete.

\section{Analysis of SGD with Mini-batch Residual Feedback}
\label{sec:Proof_Minibatch}
In this section, we analyze the query complexity of SGD with the mini-batch residual feedback. First, we make some additional assumptions.
 
\begin{asmp}
	\label{asmp:BoundedVariance_Gradient}
	When function $F(x, \xi) \in C^{1,1}$, we assume that
	\begin{equation*}
		\label{eqn:BoundedVariance}
		\mathbb{E}_\xi [ \| \nabla F(x, \xi) - \mathbb{E}[ \nabla F(x, \xi)] \|^2 ] \leq \sigma_g^2.
	\end{equation*}
\end{asmp}
Before presenting the main results, we first establish some important lemmas. 
The following lemma provides a characterization for the estimation variance of the estimator $\widetilde{g}_b(x_t)$.

\begin{lem}\label{le:variance11}
	When function $F(x, \xi) \in C^{0,0}$ with constant $L_0(\xi)$, given Assumptions~\ref{asmp:BoundedVariance} and \ref{asmp:BoundedLipschitz}, we have that
	\begin{align}
	& \mathbb{E} \|\widetilde g_b(x_t)\|^2  \leq  \frac{4(d+2)L_0^2}{\delta^2}\mathbb{E}\|x_t-x_{t-1}\|^2 & \nonumber \\
	& \quad \quad \quad \quad \quad \quad + 16L_0^2(d+4)^2  + \frac{8(d+2)\sigma^2}{\delta^2 b}.
	\end{align}
	
	Furthermore, when function $F(x, \xi) \in C^{1,1}$ with constant $L_1(\xi)$, given Assumptions~\ref{asmp:BoundedVariance}, \ref{asmp:BoundedLipschitz} and \ref{asmp:BoundedVariance_Gradient}, we have that 
	\begin{align*}
	& \mathbb{E} \|\widetilde g_b(x_t)\|^2 \leq 12L^2_1\delta^2(d+6)^3  + \frac{6 (d+2) L_0^2\eta^2}{\delta^2}\mathbb{E} \|\widetilde g_b(x_{t-1})\|^2 \nonumber
	\\ & \quad \quad \quad \quad\quad   +24(d+4)\mathbb{E}(\|\nabla f(x_t) \|^2+ \|\nabla f(x_{t-1})\|^2)  & \nonumber \\
	& \quad \quad \quad \quad \quad + \frac{48(d+4)\sigma_g^2}{b} + \frac{8(d+2)\sigma^2}{\delta^2 b}.
	\end{align*}	
\end{lem}

\begin{pf}
	When function $F(x, \xi) \in C^{0,0}$, based on the definition of $\widetilde g(x_t)$, we have
	\begin{align*}
 \|& \widetilde g_b(x_t) \|^2 = \frac{1}{\delta^2 b^2}|F(x_t+\delta u_t, \xi_{1:b}) -F(x_{t-1}+\delta u_{t-1},\xi_{1:b}) & \nonumber \\
 &  + F(x_{t-1}+\delta u_{t-1};\xi_{1:b}) - F(x_{t-1}+\delta u_{t-1},\xi_{1:b}')  |^2 \|u_t\|^2\nonumber \\
  \leq& \frac{2}{\delta^2 b^2}\big(|F(x_t+\delta u_t,\xi_{1:b}) - F(x_{t-1}+\delta u_{t-1},\xi_{1:b})|^2 + & \nonumber \\
  &   |F(x_{t-1}+\delta u_{t-1},\xi_{1:b})-F(x_{t-1}+\delta u_{t-1},\xi_{1:b}')  |^2 \big)\|u_t\|^2 \nonumber
	\\ \leq &\frac{4L_0^2}{\delta^2}\|x_t-x_{t-1}\|^2\|u_t\|^2+4L_0^2\|u_{t}- u_{t-1}\|^2\|u_t\|^2  & \nonumber \\ 
	& +\frac{2}{\delta^2 b^2} |F(x_{t-1} +\delta u_{t-1},\xi_{1:b})-F(x_{t-1}+\delta u_{t-1},\xi_{1:b}')  |^2\|u_t\|^2. & \nonumber
	\end{align*}
	Taking expectation over the above inequality yields
	\begin{align}\label{pqo}
	& \mathbb{E} \|\widetilde g_b(x_t)\|^2 \nonumber \\
	& \leq \frac{4L_0^2}{\delta^2}\mathbb{E}\big(\|x_t-x_{t-1}\|^2\|u_t\|^2\big)  +4L_0^2\mathbb{E}\big(\|u_{t}- u_{t-1}\|^2\|u_t\|^2 \big) \nonumber \\
	& \quad +\frac{2}{\delta^2 b^2} \mathbb{E}\big(|F(x_{t-1}+\delta u_{t-1},\xi_{1:b}) \nonumber \\
	& \quad \quad \quad \quad \quad - F(x_{t-1}+\delta u_{t-1},\xi_{1:b}')  |^2\|u_t\|^2 \big)\nonumber
	\\ & \leq  \frac{4L_0^2}{\delta^2}\mathbb{E}\big(\|x_t-x_{t-1}\|^2\mathbb{E}_{u_t}\|u_t\|^2 \big) \nonumber \\
	& \quad \quad \quad \quad \quad  + 8L_0^2\mathbb{E}\big(\|u_{t}\|^4+\|u_{t-1}\|^2\|u_t\|^2 \big) \nonumber
		\\& \quad  +\frac{2}{\delta^2} \mathbb{E}\big(|F(x_{t-1}+\delta u_{t-1},\xi_{1:b}) \nonumber
		\\ 
		&\quad \quad \quad \quad \quad  -F(x_{t-1}+\delta u_{t-1},\xi_{1:b}')  |^2\|u_t\|^2 \big) \nonumber \\
		& \overset{(i)}\leq  \frac{4(d+2)L_0^2}{\delta^2}\mathbb{E}\|x_t-x_{t-1}\|^2 + 8L_0^2\big((d+4)^2+(d+2)^2 \big) \nonumber
			\\& \quad +\frac{2}{\delta^2 b^2} \mathbb{E}\big(|F(x_{t-1}+\delta u_{t-1},\xi_{1:b}) \nonumber
			\\ 
			&\quad \quad \quad \quad \quad  -F(x_{t-1}+\delta u_{t-1},\xi_{1:b}')  |^2\|u_t\|^2 \big)  \nonumber \\
			&\leq \frac{4(d+2)L_0^2}{\delta^2}\mathbb{E}\|x_t-x_{t-1}\|^2 + 16L_0^2(d+4)^2 \nonumber \\
			& +\frac{4}{\delta^2b^2}\underbrace{\mathbb{E}|F(x_{t-1}+\delta u_{t-1},\xi_{1:b})- b f(x_{t-1}+\delta u_{t-1})|^2\|u_t\|^2}_{(P)} \nonumber
			\\&+ \frac{4}{\delta^2b^2}\underbrace{\mathbb{E}|F(x_{t-1}+\delta u_{t-1},\xi_{1:b}')- b f(x_{t-1}+\delta u_{t-1})|^2\|u_t\|^2}_{(Q)},
	\end{align}
	where (i) follows from Lemma 1 in \cite{nesterov2017random} that $\mathbb{E}\|u\|^p\leq (d+p)^{p/2}$ for a $d$-dimensional standard Gaussian random vector. 
Our next step is to upper-bound $(P)$ and $(Q)$ in the above inequality. For $(P)$,	conditioning on $x_{t-1}$ and $u_{t-1}$ and noting that $\xi_{1:b}$ is independent of $u_t$, we have 
\begin{align}\label{p1}
& (P) = \mathbb{E}_{\xi_{1:b}}\Big| \sum_{\xi\in \xi_{1:b}}\big(F(x_{t-1}+\delta u_{t-1},\xi) \nonumber \\
& \quad \quad \quad \quad \quad \quad \quad \quad \quad -f(x_{t-1}+\delta u_{t-1})\big)\Big|^2\mathbb{E}_{u_t}\|u_t\|^2 \nonumber
\\ \leq & \; (d+2)\mathbb{E}_{\xi_{1:b}}\Big|  \sum_{\xi\in \xi_{1:b}}\big(F(x_{t-1}+\delta u_{t-1},\xi) \nonumber \\
& \quad \quad \quad \quad \quad \quad \quad \quad \quad -f(x_{t-1}+\delta u_{t-1})\big)\Big|^2 \nonumber
\\ = &\; b(d+2) \mathbb{E}_{\xi}|F(x_{t-1}+\delta u_{t-1},\xi)-f(x_{t-1}+\delta u_{t-1})|^2  \nonumber
\\& + (d+2) \sum_{i\neq j,\xi_i,\xi_j\in \xi_{1:b}}  \langle \mathbb{E}_{\xi_i}F(x_{t-1}+\delta u_{t-1},\xi_i) \nonumber \\
& -f(x_{t-1}+\delta u_{t-1}),\mathbb{E}_{\xi_j} F(x_{t-1}+\delta u_{t-1},\xi_j) \nonumber \\
& -f(x_{t-1}+\delta u_{t-1})\rangle \nonumber
\\= & \; b(d+2) \mathbb{E}_\xi\|F(x_{t-1}+\delta u_{t-1},\xi) \nonumber \\
& -f(x_{t-1}+\delta u_{t-1})\|^2 \leq  b(d+2)\sigma^2.
\end{align}
Unconditioning on $x_{t-1}$ and $u_{t-1}$ in the above equality yields $(P)\leq b(d+2)\sigma^2$. 
 For the term $(Q)$, we have 
 \begin{align}
 & (Q) = \; \mathbb{E}\big(|F(x_{t-1}+\delta u_{t-1},\xi_{1:b}')- b f(x_{t-1}+\delta u_{t-1})|^2 \nonumber \\
 & \quad \quad \quad \quad \quad  \quad \quad \quad \quad \quad \quad \quad  \mathbb{E}_{u_t}\|u_t\|^2\, \big| x_{t-1},\xi_{1:b}',u_{t-1}\big) \nonumber
 \\ & \leq (d+2) \mathbb{E}|F(x_{t-1}+\delta u_{t-1},\xi_{1:b}')- b f(x_{t-1}+\delta u_{t-1})|^2,\nonumber
 \end{align}
 which, using an approach similar to the steps in~\eqref{p1}, yields
 \begin{align}\label{q1}
 (Q)\leq b(d+2)\sigma^2.
 \end{align}
 Combining~\eqref{pqo},~\eqref{p1} and~\eqref{q1} yields the proof when $F(x, \xi) \in C^{0,0}$.
 
 When function $F(x, \xi) \in C^{1,1}$, based on the definition of $\widetilde g_b(x_t)$, we have 
 \begin{align*}
 & \|\widetilde g_b(x_t)\|^2 = \frac{1}{\delta^2 b^2}|F(x_t+\delta u_t, \xi_{1:b}) -F(x_{t-1}+\delta u_{t-1},\xi_{1:b}) & \nonumber \\
 &  +F(x_{t-1}+\delta u_{t-1},\xi_{1:b})-F(x_{t-1}+\delta u_{t-1},\xi_{1:b}')  |^2 \|u_t\|^2 & \nonumber
 \\ & \leq \frac{2}{\delta^2 b^2} |F(x_t+\delta u_t,\xi_{1:b}) -F(x_{t-1}+\delta u_{t-1},\xi_{1:b})|^2 \|u_t\|^2 & \nonumber \\
 & + \frac{2}{\delta^2 b^2} |F(x_{t-1}+\delta u_{t-1},\xi_{1:b})-F(x_{t-1}+\delta u_{t-1},\xi_{1:b}')  |^2 \|u_t\|^2, \nonumber
 \end{align*}
 which, taking expectation and using an approach similar to~\eqref{pqo}, yields 
 \begin{align}\label{opiis}
 & \mathbb{E} \|\widetilde g(x_t)\|^2 \leq \frac{8(d+2)\sigma^2}{\delta^2 b} + \nonumber \\
 & \frac{2}{\delta^2 b^2} \underbrace{\mathbb{E}|F(x_t+\delta u_t,\xi_{1:b}) -F(x_{t-1}+\delta u_{t-1},\xi_{1:b})|^2 \|u_t\|^2}_{(P)}. 
 \end{align}
 Our next step is to upper-bound $(P)$ in the above inequality. We first divide $(P)$ into three parts such that $(P) \leq 3 \mathbb{E}\big[ (P_1) + (P_2) + (P_3) \big]$, where
 \begin{align*}
 	(P_1) = & |F(x_t + \delta u_t, \xi_{1:b}) -F(x_t, \xi_{1:b}) - \nonumber \\
 	& \inner{\delta u_t}{\nabla F(x_t,\xi_{1:b})} + \inner{\delta u_t}{\nabla F(x_t,\xi_{1:b})}|^2\|u_t\|^2, \\
 	(P_2) = & |F(x_t,\xi_{1:b}) - F(x_{t-1},\xi_{1:b})|^2\|u_t\|^2, \text{ and } \\
 	(P_3) = & |F(x_{t-1},\xi_{1:b}) - F(x_{t-1} + \delta u_{t-1},\xi_{1:b}) \nonumber \\
 	& + \inner{\delta u_{t-1}}{\nabla F(x_{t-1},\xi_{1:b})} -\inner{\delta u_{t-1}}{\nabla F(x_{t-1},\xi_{1:b})}|^2\|u_t\|^2
 \end{align*} 
 Using the assumption that $F(x;\xi)\in C^{0,0}\cap C^{1,1}$, we have
 \begin{align}
 (P_1) & \le b^2 L^2_1\delta^4\|u_t\|^6 + 2\delta^2| \inner{u_t}{\nabla F(x_t,\xi_{1:b})} |^2\|u_t\|^2, \nonumber\\
 (P_2) &\le b^2 L_0^2  \|x_{t} - x_{t-1}\|^2\|u_t\|^2, \nonumber\\ 
 (P_3) &\le b^2 L^2_1\delta^4\|u_{t-1}\|^4\|u_t\|^2  +2\delta^2 | \inner{u_{t-1}}{\nabla F(x_{t-1},\xi_{1:b})}|^2\|u_t\|^2. \nonumber
 \end{align}
 Plugging the above inequalities into $(P) \leq 3 \mathbb{E}\big[ (P_1) + (P_2) + (P_3) \big]$, we have
 \begin{align}\label{fanren}
 (P) \leq \; &3 b^2 L^2_1\delta^4\mathbb{E}\|u_t\|^6 + 6\delta^2\mathbb{E}| \inner{u_t}{\nabla F(x_t,\xi_{1:b})} |^2\|u_t\|^2 \nonumber \\
 &  + 3 b^2 L_0^2\mathbb{E} \|x_t-x_{t-1}\|^2\|u_t\|^2 \nonumber
 \\ &+ 3 b^2 L^2_1\delta^4\mathbb{E}\|u_{t-1}\|^4\|u_t\|^2  \nonumber \\
 & +6\delta^2 \mathbb{E}| \inner{u_{t-1}}{\nabla F(x_{t-1},\xi_{1:b})}|^2\|u_t\|^2, 
 \end{align}
 Based on the results in~\cite{nesterov2017random}, we have  $\mathbb{E}_u[\|u\|^p]\le (d+p)^{p/2}$, $ \mathbb{E}[\inner{u_t}{\nabla F(x_t,\xi_{1:b})}^2\|u_t\|^2] \le (d+4) \|\nabla F(x_t,\xi_{1:b})\|^2$, $\mathbb{E}[\inner{u_{t-1}}{\nabla F(x_{t-1},\xi_{1:b})}^2] \le \|\nabla F(x_t,\xi_{1:b})\|^2$, which, in conjunction with \eqref{fanren}, yields
 \begin{align}\label{ggsimida}
 & (P)\leq 6 b^2 L^2_1\delta^4(d+6)^3  + 3 b^2 (d+2) L_0^2\mathbb{E} \|x_t-x_{t-1}\|^2 \nonumber
 \\ & \quad \quad \quad \quad  +6(d+4)\delta^2\mathbb{E}\|\nabla F(x_t,\xi_{1:b}) \|^2 \nonumber \\
 & \quad \quad \quad \quad  +6(d+2)\delta^2 \mathbb{E}\|\nabla F(x_{t-1},\xi_{1:b})\|^2 \nonumber
 \\ & \overset{(i)}\leq 6 b^2 L^2_1\delta^4(d+6)^3  + 3 b^2 (d+2) L_0^2\mathbb{E} \|x_t-x_{t-1}\|^2 \nonumber \\
 & \quad \quad +12 b^2 (d+4)\delta^2\mathbb{E}\|\nabla f(x_t) \|^2  +12 b (d+4)\delta^2\sigma_g^2 \nonumber
 \\ & \quad \quad +12 b^2 (d+2)\delta^2 \mathbb{E}\|\nabla f(x_{t-1})\|^2 + 12 b (d+2)\delta^2\sigma_g^2 \nonumber
 \\  & \leq 6 b^2 L^2_1\delta^4(d+6)^3  + 3 b^2 (d+2) L_0^2\mathbb{E} \|x_t-x_{t-1}\|^2 \nonumber
 \\ & \quad \quad \quad +12 b^2 (d+4)\delta^2\mathbb{E}\|\nabla f(x_t) \|^2 + 24 b (d+4)\delta^2\sigma_g^2 \nonumber \\
 & \quad \quad \quad +12 b^2 (d+2)\delta^2 \mathbb{E}\|\nabla f(x_{t-1})\|^2.  
 \end{align}
 Combining~\eqref{ggsimida} and \eqref{opiis} yields
 \begin{align}
 & \mathbb{E} \|\widetilde g_b(x_t)\|^2 \leq 12L^2_1\delta^2(d+6)^3  + \frac{6 (d+2) L_0^2\eta^2}{\delta^2}\mathbb{E} \|\widetilde g_b(x_{t-1})\|^2 \nonumber
 \\ &  +24(d+4)\mathbb{E}(\|\nabla f(x_t) \|^2+ \|\nabla f(x_{t-1})\|^2) + \frac{48(d+4)\sigma_g^2}{b} \nonumber \\
 & + \frac{8(d+2)\sigma^2}{\delta^2 b},
 \end{align}
 which finishes the proof. 
\end{pf}

First, we analyze the convergence when the problem is non-smooth. 
Based on Lemma~\ref{le:variance11}, we provide an upper bound on $\mathbb{E}\|x_{t+1}-x_t\|^2$.
\begin{lem}\label{le:online_ind}
Suppose Assumptions~\ref{asmp:BoundedVariance} and \ref{asmp:BoundedLipschitz} are satisfied. Then, we have 
\begin{align}
\mathbb{E}\|x_{t+1}-x_t\|^2 \leq \beta_1^t\Big( \mathbb{E}\|x_{1}-x_{0}\|^2 -\frac{\beta_2}{1-\beta_1}\Big)+\frac{\beta_2}{1-\beta_1}, \nonumber
\end{align}
where $\beta_1 =\frac{4\eta^2(d+2)L_0^2}{\delta^2} $ and $\beta_2 = 16\eta ^2 L_0^2(d+4)^2  +\frac{8\eta^2(d+2)\sigma^2}{\delta^2 b}$. 
\end{lem}
\begin{pf}
Based on the update that $x_{t+1}-x_t=-\eta \widetilde g_b(x_t)$ and Lemma~\ref{le:variance11}, we have 
\begin{align*}
& \mathbb{E}\|x_{t+1}-x_t\|^2 =\eta^2 \|\widetilde g_b(x_t)\|^2 \\
& \leq \eta^2 \Big( \frac{4(d+2)L_0^2}{\delta^2}\mathbb{E}\|x_t-x_{t-1}\|^2 + 16L_0^2(d+4)^2  \nonumber \\
&  \quad \quad \quad \quad \quad \quad \quad \quad \quad \quad \quad \quad \quad \quad \quad + \frac{8(d+2)\sigma^2}{\delta^2 b}  \Big)
\\ & =\frac{4\eta^2(d+2)L_0^2}{\delta^2} \mathbb{E}\|x_t-x_{t-1}\|^2 +16\eta ^2 L_0^2(d+4)^2 \nonumber \\
& \quad \quad \quad \quad \quad \quad \quad \quad \quad \quad \quad \quad \quad \quad \quad  +\frac{8\eta^2(d+2)\sigma^2}{\delta^2 b}
\\ & = \beta_1 \mathbb{E}\|x_t-x_{t-1}\|^2+\beta_2.
\end{align*}
Then, telescoping the above inequality yields the proof. 
\end{pf}
\subsection*{Nonsmooth Nonconvex Geometry}
Based on the above lemmas, we next provide the  convergence and complexity analysis for our proposed algorithm for the case where $F(x;\xi)$ is nonconvex and belongs to $C^{0,0}$. 
\begin{thm}
Suppose Assumptions~\ref{asmp:BoundedVariance} and \ref{asmp:BoundedLipschitz} are satisfied. Choose $\eta= \frac{\epsilon_f^{1/2}}{2(d+2)^{3/2}T^{1/2}L^2_0}, \delta = \frac{\epsilon_f}{(d+2)^{1/2}L_0}$ and $b=\frac{\sigma^2}{\epsilon_f^2}\geq 1$ for certain $\epsilon_f<1$.  Then, we have 
$\mathbb{E}\|\nabla f_\delta(x_\zeta)\|^2
\leq \mathcal{O}\left( \frac{d^{3/2}}{\epsilon_f^{1/2}\sqrt{T}} \right) $ with the approximation error $|f_\delta (x_\zeta)-f(x_\zeta)|<\theta$, where $\zeta$ is uniformly sampled from $\{0, 1, \dots, T-1 \}$. 
Then, to achieve an $\epsilon$-accurate stationary point of  $f_\delta$, the corresponding total function query complexity is given by 
\begin{align}
Tb =  \mathcal{O}\left( \frac{\sigma^2 d^{3}}{\epsilon_f^{3}\epsilon^{2}} \right).
\end{align}
\end{thm} 
\begin{pf}
	Recall that $f_\delta(x)=\mathbb{E}_{u}f(x+\delta u)$ is a smoothed approximation of  $f(x)$, where $u\in\mathbb{R}^d$ is a standard Gaussian random vector. Based on Lemma 2 in~\cite{nesterov2017random},  we have $f_\delta \in C^{1,1}$ with gradient-Lipschitz constant $L_\delta$ satisfying $L_\delta\leq \frac{d^{1/2}}{\delta} L_0$, and thus 
\begin{align*}
& f_\delta(x_{t+1}) \leq f_\delta(x_t) +\langle \nabla f_\delta(x_t), x_{t+1}-x_t\rangle +\frac{L_\delta}{2} \|x_{t+1}-x_t\|^2
\\ & \leq f_\delta(x_t) +\langle \nabla f_\delta(x_t), x_{t+1}-x_t\rangle +\frac{d^{1/2}L_0}{2\delta} \|x_{t+1}-x_t\|^2
\\ & = f_\delta(x_t) -\eta \langle \nabla f_\delta(x_t), \widetilde g(x_t)\rangle + \frac{d^{1/2}L_0}{2\delta} \|x_{t+1}-x_t\|^2.
\end{align*}
Taking expectation over the above inequality and using $\mathbb{E}(\widetilde g(x_t) | x_t)=\nabla f_\delta(x_t)$, we have
\begin{align*}
\mathbb{E} f_\delta (x_{t+1}) \leq & \; \mathbb{E} f_\delta (x_t) -\eta \mathbb{E}\|\nabla f_\delta(x_t)\|^2 \nonumber \\
& \quad \quad \quad \quad \quad + \frac{d^{1/2}L_0}{2\delta} \mathbb{E}\|x_{t+1}-x_t\|^2, \nonumber
\end{align*}
which, in conjunction with Lemma~\ref{le:online_ind}, yields
\begin{align}
& \mathbb{E} f_\delta (x_{t+1}) \leq \mathbb{E} f_\delta (x_t) -\eta \mathbb{E}\|\nabla f_\delta(x_t)\|^2  \nonumber \\
& + \frac{d^{1/2}L_0}{2\delta} \beta_1^t\Big( \mathbb{E}\|x_{1}-x_{0}\|^2 -\frac{\beta_2}{1-\beta_1}\Big)+\frac{d^{1/2}L_0}{2\delta} \frac{\beta_2}{1-\beta_1}. \nonumber
\end{align}
Telescoping the above inequality over $t$ from $0$ to $T-1$ yields
\begin{align*}
& \sum_{t=0}^{T-1}\eta \mathbb{E}\|\nabla f_\delta (x_t)\|^2 & \nonumber \\
& \leq f_\delta(x_0)- \inf_x f_\delta(x) + \frac{d^{1/2}L_0 T}{2\delta} \frac{\beta_2}{1-\beta_1} \\
& \quad \quad \quad  + \frac{d^{1/2}L_0}{2\delta} \Big( \mathbb{E}\|x_{1}-x_{0}\|^2 -\frac{\beta_2}{1-\beta_1}\Big)\sum_{t=0}^{T-1}\beta_1^t
\\ & = f_\delta(x_0)- \inf_x f_\delta(x) + \frac{d^{1/2}L_0 T}{2\delta} \frac{\beta_2}{1-\beta_1} \\
&  \quad \quad \quad + \frac{d^{1/2}L_0}{2\delta} \Big( \mathbb{E}\|x_{1}-x_{0}\|^2 -\frac{\beta_2}{1-\beta_1}\Big)\frac{1-\beta_1^T}{1-\beta_1}
\\ & \leq  f(x_0)- \inf_x f(x) + 2 \delta L_0 d^{1/2}+ \frac{d^{1/2}L_0 T}{2\delta} \frac{\beta_2}{1-\beta_1} \\
& \quad \quad \quad  + \frac{d^{1/2}L_0}{2\delta} \Big( \mathbb{E}\|x_{1}-x_{0}\|^2 -\frac{\beta_2}{1-\beta_1}\Big)\frac{1-\beta_1^T}{1-\beta_1},
\end{align*}
where the last inequality follows from Equation (3.11) in~\cite{ghadimi2013stochastic} and Equation (18) in~\cite{nesterov2017random}. Choose $\eta= \frac{\epsilon_f^{1/2}}{2(d+2)^{3/2}T^{1/2}L^2_0}$ and $\delta = \frac{\epsilon_f}{(d+2)^{1/2}L_0}$ with certain $\epsilon_f<1$, and set $T>\frac{1}{2\epsilon_f d}$. 
Then, we have $\beta_1 = \frac{1}{\epsilon_f(d+2)T}<\frac{1}{2}$, 
and thus the above inequality yields 
\begin{align}
\sum_{t=0}^{T-1}\eta \mathbb{E}\|& \nabla f_\delta (x_t)\|^2 \leq f(x_0)- \inf_x f(x) + 2 \delta L_0 d^{1/2} \nonumber \\
& + \frac{\beta_2 d^{1/2}L_0 T}{\delta} + \frac{d^{1/2}L_0\eta^2}{\delta}\mathbb{E}\|\widetilde g_b(x_0)\|^2. \nonumber
\end{align}
Choosing $\zeta$  from $0,1,...,T-1$ uniformly at random, and rearranging  the above inequality, we have
\begin{align}
& \mathbb{E}\|\nabla f_\delta(x_\zeta)\|^2 \leq \frac{f(x_0)- \inf_x f(x)  }{\eta T} +\frac{2 \delta L_0 d^{1/2} }{\eta T} \nonumber \\
& \quad \quad \quad \quad \quad \quad  + \frac{d^{1/2}L_0\eta}{\delta T}\mathbb{E}\|\widetilde g(x_0)\|^2  + \frac{16\eta (d+4)^2d^{1/2}L_0^3 }{\delta}  \nonumber \\
& \quad \quad \quad  \quad \quad \quad +   \frac{ 8\eta (d+2)d^{1/2}L_0 \sigma^2}{\delta^3 b} \nonumber \\
\leq & \mathcal{O}\left( \frac{1}{\eta T} + \frac{\delta L_0 d^{1/2}}{\eta T} +   \frac{d^{1/2}L_0\eta}{\delta T} + \frac{\eta d^{5/2}L_0^3 }{\delta}    +   \frac{ \eta d^{3/2}L_0 \sigma^2}{\delta^3 b}\right), \nonumber
\end{align}
which, in conjunction with $\eta= \frac{\epsilon_f^{1/2}}{2(d+2)^{3/2}T^{1/2}L^2_0}, \delta = \frac{\epsilon_f}{(d+2)^{1/2}L_0}$ and $b=\frac{\sigma^2}{\epsilon_f^2}$, yields
\begin{align}
& \mathbb{E}\|\nabla f_\delta(x_\zeta)\|^2 & \nonumber \\
& \leq \mathcal{O}\left( \frac{1}{\eta T} + \frac{\delta L_0 d^{1/2}}{\eta T } +   \frac{d^{1/2}L_0\eta}{\delta T} + \frac{\eta d^{5/2}L_0^3 }{\delta}    +   \frac{ \eta d^{3/2}L_0 \sigma^2}{\delta^3 b}\right) \nonumber
\\ & \leq \mathcal{O} \left(  \frac{d^{3/2}L^2_0}{\epsilon_f^{1/2}\sqrt{T}} + \frac{\epsilon_f^{1/2}d^{3/2}L_0^2}{\sqrt{T}} + \frac{1}{\epsilon_f^{1/2}d^{1/2}T^{3/2}} + \frac{d^{3/2}\sigma^2}{\epsilon_f^{5/2}\sqrt{T} b}\right) \nonumber \\
& \leq \mathcal{O}\left( \Big(1+\frac{\sigma^2}{\epsilon_f^{2}b}\Big)\frac{d^{3/2}}{\epsilon_f^{1/2}\sqrt{T}} \right) \leq \mathcal{O}\left( \frac{d^{3/2}}{\epsilon_f^{1/2}\sqrt{T}} \right). \nonumber
\end{align}
Based on $ \delta = \frac{\epsilon_f}{(d+2)^{1/2}L_0}$ and Equation (18) in~\cite{nesterov2017random}, we have $|f_\delta(x)-f(x)| < \epsilon_f$. Then, to achieve an $\epsilon$-accurate stationary point of the smoothed function $f_\delta$ with approximation error $|f_\delta(x)-f(x)| <\epsilon_f, \epsilon_f<1$ , we need 
$T\leq \mathcal{O}(\epsilon_f^{-1}d^{3}\epsilon^{-2})$, and thus the corresponding total function query complexity is given by 
\begin{align}
Tb =  \mathcal{O}\left( \frac{\sigma^2 d^{3}}{\epsilon_f^{3}\epsilon^{2}} \right).
\end{align}
Then, the proof is complete. 
	\end{pf}
\subsection*{Nonsmooth Convex Geometry}
In this part, we provide the convergence and complexity  analysis for our proposed algorithm for the case where $F(x;\xi)$ is convex and belongs to $C^{0,0}$. 
\begin{thm}
Suppose Assumptions~\ref{asmp:BoundedVariance} and \ref{asmp:BoundedLipschitz} are satisfied and $\mathbb{E}\|\widetilde g_b(x_{0})\|^2\leq M d^2 T$ for certain constant $M>0$. Choose $\eta =\frac{1}{(d+2)\sqrt{T}L_0}$, $\delta = \frac{(d+2)^{1/2}}{\sqrt{T}}$, $b=\frac{\sigma^2 T}{d^2}$, and $T>d^2$.  Then, we have 
$\mathbb{E}\big(f(x_\zeta) - \inf_x f(x) \big) 
\leq  \mathcal{O}\big(\frac{d}{\sqrt{T}}\big) $.
Then, to achieve an $\epsilon$-accurate solution of  $f(x)$, the corresponding total function query complexity is given by 
\begin{align}
Tb = \mathcal{O} \Big( \frac{\sigma^2 d^2}{\epsilon^{4}}\Big).
\end{align}
	\end{thm}
\begin{pf}
	Let $x^*$ be a  minimizer of the function $f$, i.e. $x^*=\arg\min_x f(x)$. Then, we have 
\begin{align*}
& \|x_{t+1} - x^*\|^2 = \|x_t - \eta \widetilde{g}_b(x_t) -x^*\|^2 \nonumber\\
&= \|x_{t} - x^*\|^2 -2\eta \inner{\widetilde{g}_b(x_t)}{x_t - x^*} + \mathbb{E}\|x_{t+1}-x_{t}\|^2.
\end{align*}
Telescoping the above inequality over $t$ from $0$ to $T-1$ yields that
\begin{align*}
\|x_T - x^*\|^2 = & \; \|x_0 - x^*\|^2 - 2\eta \sum_{t=0}^{T-1} \inner{\widetilde{g}_b(x_t)}{x_t - x^*} \\
& +  \sum_{t=0}^{T-1} \mathbb{E}\|x_{t+1}-x_{t}\|^2. 
\end{align*}
Taking expectation in the above equality using the fact that $\mathbb{E}[\widetilde{g}_b(x_t)|x_t] = \nabla f_\delta(x_t)$, we further obtain that
\begin{align}\label{eq: wocaca}
& \mathbb{E}\|x_T - x^*\|^2 = \|x_0 - x^*\|^2 - 2\eta \sum_{t=0}^{T-1} \mathbb{E}\inner{\nabla f_\delta(x_t)}{x_t - x^*} \nonumber \\
& \quad \quad \quad \quad \quad \quad \quad \quad \quad \quad \quad \quad \quad \quad +  \sum_{t=0}^{T-1} \mathbb{E}\|x_{t+1}-x_{t}\|^2 \nonumber\\
&\overset{(i)}{\le} \|x_0 - x^*\|^2 - 2\eta \sum_{t=0}^{T-1} \mathbb{E}\big(f_\delta(x_t) - f_\delta(x^*) \big)  & \nonumber \\
& \quad \quad \quad \quad \quad \quad \quad \quad \quad \quad \quad \quad \quad \quad + \sum_{t=0}^{T-1} \mathbb{E}\|x_{t+1}-x_{t}\|^2 \nonumber\\
&\overset{(ii)}{\le} \|x_0 - x^*\|^2 - 2\eta \sum_{t=0}^{T-1} \mathbb{E}\big(f(x_t) - f(x^*) \big) +4\eta\delta L_0 \sqrt{d}T \nonumber \\
& \quad \quad \quad \quad \quad \quad \quad \quad \quad \quad \quad + \sum_{t=0}^{T-1} \mathbb{E}\|x_{t+1}-x_{t}\|^2, 
\end{align} 
where (i) follows from the convexity of $f_\delta$ and (ii) uses the fact that $|f_\delta(x) - f(x)|\le \delta L_0 \sqrt{d}$. Then,  rearranging the above inequality yields
\begin{align*}
& \frac{1}{T}\sum_{t=0}^{T-1}\mathbb{E}\big(f(x_t) - f(x^*) \big) \leq \frac{\|x_0-x^*\|^2}{\eta T}+ 4\delta L_0 \sqrt{d} \\
& \quad \quad \quad \quad \quad \quad \quad \quad \quad \quad \quad  + \frac{1}{\eta T} \sum_{t=0}^{T-1}\mathbb{E}\|x_{t+1}-x_{t}\|^2 \nonumber
\\& \overset{(i)}\leq \frac{\|x_0-x^*\|^2}{\eta T}+ 4\delta L_0 \sqrt{d} \\
& \quad \quad + \frac{1}{\eta T} \sum_{t=0}^{T-1} \Big( \beta_1^t\Big( \mathbb{E}\|x_{1}-x_{0}\|^2 -\frac{\beta_2}{1-\beta_1}\Big)+\frac{\beta_2}{1-\beta_1}\Big) \nonumber
\\ & \leq \frac{\|x_0-x^*\|^2}{\eta T}+ 4\delta L_0 \sqrt{d} \\
& \quad \quad + \frac{1}{\eta T} \frac{1-\beta_1^T}{1-\beta_1} \Big(\mathbb{E}\|x_{1}-x_{0}\|^2 -\frac{\beta_2}{1-\beta_1}\Big)+\frac{\beta_2}{\eta(1-\beta_1)},
\end{align*} 
where (i) follows from Lemma~\ref{le:online_ind} with $\beta_1 =\frac{4\eta^2(d+2)L_0^2}{\delta^2} $ and $\beta_2 = 16\eta ^2 L_0^2(d+4)^2  +\frac{8\eta^2(d+2)\sigma^2}{\delta^2 b}$. Recalling $\eta =\frac{1}{(d+2)\sqrt{T}L_0}$, $\delta = \frac{(d+2)^{1/2}}{\sqrt{T}}$, $b=\frac{\sigma^2 T}{d^2}$, and $T>d^2$, we have $\beta_1 < 1/2$, and the above inequality yields
\begin{align*}
& \frac{1}{T}\sum_{t=0}^{T-1}\mathbb{E}\big(f(x_t) - f(x^*) \big)  \leq \frac{\|x_0-x^*\|^2}{\eta T}+ 4\delta L_0 \sqrt{d} \nonumber
\\  & + \frac{2\eta}{T} \mathbb{E}\|\widetilde g_b(x_{0})\|^2+32\eta  L_0^2(d+4)^2   +\frac{16\eta(d+2)\sigma^2}{\delta^2 b} \\
& \leq\mathcal{O}\left( \frac{d}{\sqrt{T}}+ \frac{ \mathbb{E}\|\widetilde g_b(x_{0})\|^2}{dT^{3/2}}\right), 
\end{align*}
which, combined with $\mathbb{E}\|\widetilde g_b(x_{0})\|^2\leq M d^2 T$ for constant $M$ and choosing $\zeta$ from $0,..,T-1$ uniformly at random, yields 
\begin{align}
\mathbb{E}\big(f(x_\zeta) - f(x^*) \big)  \leq  \mathcal{O}\Big(\frac{d}{\sqrt{T}}\Big). \nonumber
\end{align}
To achieve an $\epsilon$-accurate solution, i.e., $\mathbb{E}\big(f(x_\zeta) - f(x^*) \big) <\epsilon$, we need $T=\mathcal{O}(d^2\epsilon^{-2})$, and hence the corresponding function query complexity is given by 
\begin{align}
Tb\leq \mathcal{O} \Big( \sigma^2 d^2\epsilon^{-4}\Big), \nonumber 
\end{align}
which finishes the proof. 
	\end{pf}

\subsection{Analysis  in Smooth Setting} 	

In this section, we  provide the convergence and complexity analysis for the proposed gradient estimator when function $F(x, \xi) \in C^{1,1}$

\subsection*{Smooth Nonconvex Geometry}
In this part, we provide the  convergence and complexity analysis for the proposed gradient estimator  for the case where $F(x;\xi)$ is nonconvex and belongs to $C^{0,0}\cap C^{1,1}$. 
\begin{thm}
Suppose Assumptions~\ref{asmp:BoundedVariance}, \ref{asmp:BoundedLipschitz} and \ref{asmp:BoundedVariance_Gradient} are satisfied and $\mathbb{E}\|\widetilde g_b(x_0)\|^2 \leq M T d^{8/3}$ for certain constant $M>0$. Choose $\eta = \frac{1}{4(d+2)^{4/3}\sqrt{T}\max(L_0,L_1)}< \frac{1}{8L_1}, \delta = \frac{1}{(d+2)^{5/6} T^{1/4}}$ and $ b=\max\big(\sigma^2,\frac{\sigma_g^2}{\sqrt{T}d^{5/3}}\big)\sqrt{T}$.  Then, we have 
$\mathbb{E}\|\nabla f_\delta(x_\zeta)\|^2
\leq \mathcal{O}\Big( \frac{d^{4/3}}{\sqrt{T}}\Big)$.
Then, to achieve an $\epsilon$-accurate stationary point of  $f$, the  total function query complexity is given by 
\begin{align}
Tb =  \mathcal{O}\big(\sigma^2d^4\epsilon^{-3}+\sigma_g^2d\epsilon^{-2}\big). \nonumber 
\end{align}
\end{thm} 
\begin{pf}
Based on Equation (12) in~\cite{nesterov2017random}, the smoothed function $f_\delta\in C^{1,1}$ with gradient-Lipschitz constant less than $L_1$. Then, we have 
\begin{align*}
f_\delta(x_{t+1}) \leq & f_\delta(x_t) +\langle \nabla f_\delta(x_t), x_{t+1}-x_t\rangle +\frac{L_1}{2} \|x_{t+1}-x_t\|^2
\\ = & f_\delta(x_t) -\eta \langle \nabla f_\delta(x_t), \widetilde g_b(x_t)\rangle +\frac{L_1}{2}\|x_{t+1}-x_t\|^2.
\end{align*}
Let $\alpha= \frac{6 (d+2) L_0^2\eta^2}{\delta^2}$, $\beta=12L^2_1\delta^2(d+6)^3  + \frac{48(d+4)\sigma_g^2}{b}+ \frac{8(d+2)\sigma^2}{\delta^2 b}$ and $p_{t-1}=24(d+4)\mathbb{E}(\|\nabla f(x_t) \|^2+ \|\nabla f(x_{t-1})\|^2)$. Then, telescoping the bound in Lemma~\ref{le:variance11} in the smooth case yields
\begin{align}\label{ggsmoiis}
& \mathbb{E} \|\widetilde g_b(x_t)\|^2  \leq \alpha^t \mathbb{E}\|\widetilde g_b(x_0)\|^2  +   \sum_{j=0}^{t-1}\alpha^{t-1-j}p_j + \beta\sum_{j=0}^{t-1}\alpha^j & \nonumber \\
& \leq \alpha^t \mathbb{E}\|\widetilde g_b(x_0)\|^2  +   \sum_{j=0}^{t-1}\alpha^{t-1-j}p_j + \frac{\beta(1-\alpha^t)}{1-\alpha}.
\end{align}
Taking expectation over the above inequality and using $\mathbb{E}(\widetilde g_b(x_t) | x_t)=\nabla f_\delta(x_t)$, we have
\begin{align*}
\mathbb{E} f_\delta (x_{t+1}) \leq \mathbb{E} f_\delta (x_t) -\eta \mathbb{E}\|\nabla f_\delta(x_t)\|^2 + \frac{L_1\eta^2}{2} \mathbb{E}\|\widetilde g_b(x_t)\|^2.
\end{align*}
Telescoping the above inequality over $t$ from $0$ to $T-1$ yields
\begin{align}
& \mathbb{E} f_\delta (x_{k}) \leq f_\delta (x_0) -\eta \sum_{t=0}^{T-1}\mathbb{E}\|\nabla f_\delta(x_t)\|^2 \nonumber \\
& \quad \quad \quad \quad \quad \quad \quad \quad \quad \quad \quad \quad + \frac{L_1\eta^2}{2} \sum_{t=0}^{T-1} \mathbb{E}\|\widetilde g_b(x_t)\|^2\nonumber
\\ & \overset{(i)}\leq f_\delta (x_0) -\eta \sum_{t=0}^{T-1}\mathbb{E}\|\nabla f_\delta(x_t)\|^2 + \frac{L_1\eta^2}{2}\mathbb{E}\|\widetilde g_b(x_0)\|^2 & \nonumber \\
& +\frac{L_1\eta^2}{2} \sum_{t=1}^{T-1} \left( \alpha^t \mathbb{E}\|\widetilde g_b(x_0)\|^2  +   \sum_{j=0}^{t-1}\alpha^{t-1-j}p_j + \frac{\beta(1-\alpha^t)}{1-\alpha}   \right)\nonumber
\\& = f_\delta (x_0) -\eta \sum_{t=0}^{T-1}\mathbb{E}\|\nabla f_\delta(x_t)\|^2 + \frac{1-\alpha^T}{1-\alpha}\frac{L_1\eta^2}{2}\mathbb{E}\|\widetilde g_b(x_0)\|^2  & \nonumber \\
& \quad \quad \quad +\frac{L_1\eta^2}{2} \sum_{j=0}^{T-2}\sum_{t=0}^{T-2-j}\alpha^{t}p_j +\frac{L_1\eta^2}{2} \sum_{t=1}^{T-1}\frac{\beta(1-\alpha^t)}{1-\alpha}  \nonumber
\end{align}
Since $\sum_{j=0}^{T-2}\sum_{t=0}^{T-2-j}\alpha^{t}p_j \leq \sum_{j=0}^{T-2}\sum_{t=0}^{T-2}\alpha^{t}p_j $, we have that
\begin{align}
&  \mathbb{E} f_\delta (x_{k}) \leq f_\delta (x_0) -\eta \sum_{t=0}^{T-1}\mathbb{E}\|\nabla f_\delta(x_t)\|^2 \nonumber \\
&  \quad \quad \quad + \frac{1-\alpha^T}{1-\alpha}\frac{L_1\eta^2}{2}\mathbb{E}\|\widetilde g_b(x_0)\|^2 +\frac{L_1\eta^2}{2} \sum_{j=0}^{T-2}\sum_{t=0}^{T-2}\alpha^{t}p_j & \nonumber
\\ 
& \quad \quad \quad + \frac{L_1\eta^2}{2} \sum_{t=1}^{T-1}\frac{\beta(1-\alpha^t)}{1-\alpha} \nonumber \\
& \leq f_\delta (x_0) -\eta \sum_{t=0}^{T-1}\mathbb{E}\|\nabla f_\delta(x_t)\|^2 + \frac{1-\alpha^T}{1-\alpha}\frac{L_1\eta^2}{2}\mathbb{E}\|\widetilde g_b(x_0)\|^2 & \nonumber \\
& \quad \quad \quad +\frac{L_1\eta^2}{2} \frac{1-\alpha^{T-1}}{1-\alpha}\sum_{t=0}^{T-2}p_t +\frac{L_1\eta^2}{2} \sum_{t=1}^{T-1}\frac{\beta(1-\alpha^t)}{1-\alpha}, \nonumber
\end{align}
where (i) follows from~\eqref{ggsmoiis}. Choose $\eta = \big(4(d+2)^{4/3}\sqrt{T}\big)^{-1}$ $\max(L_0,L_1)^{-1}$ and $\delta = \frac{1}{(d+2)^{5/6} T^{1/4}}$. Then, we have $\alpha\leq \frac{3}{8}<\frac{1}{2}$, and the above inequality yields
\begin{align}
& \mathbb{E} f_\delta (x_{k})  \leq f_\delta (x_0) -\eta \sum_{t=0}^{k-1}\mathbb{E}\|\nabla f_\delta(x_t)\|^2 + L_1\eta^2\mathbb{E}\|\widetilde g_b(x_0)\|^2 \nonumber \\
& \quad \quad \quad \quad \quad +L_1\eta^2\sum_{t=0}^{T-2}p_t +L_1\eta^2 T\beta. \nonumber
\end{align}
Rearranging the above inequality and using $|f_\delta(x) - f|\leq \frac{\delta^2}{2}L_1d$ and $\|\nabla f_\delta(x)-\nabla f(x)\|\leq \frac{\delta}{2} L_1 (d+3)^{3/2}$ proved in~\cite{nesterov2017random}, we have 
\begin{align}
\mathbb{E} & f(x_{k})  \leq f (x_0) +\delta^2 L_1 d-\frac{\eta}{2} \sum_{t=0}^{T-1}\mathbb{E}\|\nabla f(x_t)\|^2  & \nonumber \\
& \quad \quad \quad \quad +\frac{\eta T}{4}\delta^2 L_1^2 (d+3)^3 + L_1\eta^2\mathbb{E}\|\widetilde g_b(x_0)\|^2  \nonumber \\
&  \quad \quad \quad \quad +L_1\eta^2\sum_{t=0}^{T-2}p_t +L_1\eta^2 T \beta \nonumber
\\ \leq & f(x_0) +\delta^2 L_1 d-\frac{\eta}{2} \sum_{t=0}^{T-1}\mathbb{E}\|\nabla f(x_t)\|^2 +\frac{\eta T}{4}\delta^2 L_1^2 (d+3)^3 & \nonumber \\
& + L_1\eta^2\mathbb{E}\|\widetilde g_b(x_0)\|^2 +2L_1\eta^2 \sum_{t=0}^{T-1} \|\nabla f(x_t)\|^2+L_1\eta^2 T \beta. \nonumber
\end{align}
Choosing $\zeta$ from $0,...,T-1$ uniformly at random, we obtain from the above inequality that 
\begin{align*}
& \big(  \frac{1}{2} - 2L_1\eta\big) \mathbb{E} \|\nabla f(x_\zeta)\|^2 \leq  \frac{f(x_0)-\inf_x f(x)}{\eta T} + \frac{\delta^2 L_1 d}{\eta T} \nonumber \\
& \quad \quad \quad \quad +\frac{L_1^2}{4}\delta^2 (d+3)^3 +\frac{L_1\eta\mathbb{E}\|\widetilde g_b(x_0)\|^2}{T} + L_1 \eta \beta, 
\end{align*}
which, in conjunction with $\eta = \frac{1}{4(d+2)^{4/3}\sqrt{T}\max(L_0,L_1)}< \frac{1}{8L_1}, \delta = \frac{1}{(d+2)^{5/6}T^{1/4}}, \beta=12L^2_1\delta^2(d+6)^3  + \frac{48(d+4)\sigma_g^2}{b}+ \frac{8(d+2)\sigma^2}{\delta^2 b}, b=\max\big(\sigma^2,\frac{\sigma_g^2}{\sqrt{T}d^{5/3}}\big)\sqrt{T}$ and $\mathbb{E}\|\widetilde g_b(x_0)\|^2 \leq M T d^{8/3}$, yields 
\begin{align*}
\mathbb{E} \|\nabla f(x_\zeta)\|^2 \leq & \; \mathcal{O}\Big( \frac{d^{4/3}}{\sqrt{T}}+ \frac{d^{2/3}}{T}+ \frac{d^{4/3}}{\sqrt{T}}+ \frac{d^{4/3}}{\sqrt{T}}+\frac{1}{T} \\
& +\frac{\sigma_g^2}{d^{1/3}b\sqrt{T}}+\frac{d^{4/3}\sigma^2}{b} \Big) \leq \mathcal{O}\Big( \frac{d^{4/3}}{\sqrt{T}}\Big).
\end{align*}
Then, to achieve an $\epsilon$-accurate stationary point of  function $f$, i.e., $\mathbb{E}\|\nabla f(x_\zeta)\|^2<\epsilon$, we need $T =\mathcal{O}(d^{8/3}\epsilon^{-2})$the total number of function query is given by $Tb\leq \mathcal{O}\big(\sigma^2d^4\epsilon^{-3}+\sigma_g^2d\epsilon^{-2}\big)$.
\end{pf}
	
\subsection*{Smooth Convex Geometry}
In this part, we provide the  convergence and complexity analysis for the proposed gradient estimator  for the case where $F(x;\xi)$ is convex and belongs to $C^{0,0}\cap C^{1,1}$. 
\begin{thm}
Suppose Assumptions~\ref{asmp:BoundedVariance}, \ref{asmp:BoundedLipschitz} and \ref{asmp:BoundedVariance_Gradient} are satisfied and $\mathbb{E}\|\widetilde g_b(x_0)\|^2 \leq M T d^{2}$ for certain constant $M>0$. Choose $\eta = \frac{1}{192(d+2)\sqrt{T}\max(L_0,L_1)}$ and $\delta^2 =\frac{1}{\sqrt{T}}$ and $b=\max\big(\frac{\sigma_g^2}{\sqrt{T}},\sigma^2\big)\sqrt{T}/d$.  Then, we have 
$\mathbb{E}\|\nabla f_\delta(x_\zeta)\|^2
\leq \mathcal{O} \Big( \frac{d}{\sqrt{T}}+ \frac{d^2}{T}\Big)$.
Then, to achieve an $\epsilon$-accurate stationary point of  $f$, the  total function query complexity is given by 
\begin{align}
Tb =  \mathcal{O}\big(\sigma^2d^2\epsilon^{-3}+\sigma_g^2d\epsilon^{-2}\big). \nonumber
\end{align}
\end{thm} 
\begin{pf}
Using an approach similar to \eqref{eq: wocaca}, we have 
\begin{align*}
\mathbb{E}\|x_T - x^*\|^2
{\le} & \; \|x_0 - x^*\|^2 - 2\eta \sum_{t=0}^{T-1} \mathbb{E}\big(f(x_t) - f(x^*) \big) \\
& +2\eta\delta^2 L_1 d T+ \sum_{t=0}^{T-1}\eta^2 \mathbb{E}\|\widetilde g_b(x_{t})\|^2, 
\end{align*}
where the last inequality follows from Equation (19) in~\cite{nesterov2017random}. 
Let $\alpha= \frac{6 (d+2) L_0^2\eta^2}{\delta^2}$, $\beta=12L^2_1\delta^2(d+6)^3  + \frac{48(d+4)\sigma_g^2}{b}+ \frac{8(d+2)\sigma^2}{\delta^2 b}$ and $p_{t-1}=24(d+4)\mathbb{E}(\|\nabla f(x_t) \|^2+ \|\nabla f(x_{t-1})\|^2)$. Then, 
combining the above inequality with \eqref{ggsmoiis} yields
\begin{align}
& \mathbb{E}\|x_T - x^*\|^2  \leq \|x_0 - x^*\|^2 \nonumber \\
&  \quad -2\eta \sum_{t=0}^{T-1} \mathbb{E}\big(f(x_t) - f(x^*) \big) +2\eta\delta^2 L_1 d T + \eta^2 \mathbb{E} \|\widetilde g_b(x_0)\|^2 \nonumber
\\& \quad  + \sum_{t=1}^{T-1}\eta^2\Big(\alpha^t \mathbb{E}\|\widetilde g_b(x_0)\|^2  +   \sum_{j=0}^{t-1}\alpha^{t-1-j}p_j + \frac{\beta(1-\alpha^t)}{1-\alpha} \Big) \nonumber
\\& \leq \|x_0 - x^*\|^2 - 2\eta \sum_{t=0}^{T-1} \mathbb{E}\big(f(x_t) - f(x^*) \big) +2\eta\delta^2 L_1 d T \nonumber
\\& + \eta^2\frac{1-\alpha^T}{1-\alpha} \mathbb{E} \|\widetilde g_b(x_0)\|^2  + \eta^2 \sum_{j=0}^{T-2}\sum_{t=0}^{T-2}\alpha^{t}p_j  + \eta^2 \sum_{t=1}^{T-1}\frac{\beta(1-\alpha^t)}{1-\alpha} \nonumber
\\& \leq \|x_0 - x^*\|^2 - 2\eta \sum_{t=0}^{T-1} \mathbb{E}\big(f(x_t) - f(x^*) \big) +2\eta\delta^2 L_1 d T \nonumber
\\& + \eta^2\frac{1-\alpha^T}{1-\alpha} \mathbb{E} \|\widetilde g_b(x_0)\|^2 + \eta^2 \sum_{t=1}^{T-1}\frac{\beta(1-\alpha^t)}{1-\alpha} \nonumber \\
&  + 24(d+4)\eta^2 \frac{1-\alpha^{T-1}}{1-\alpha}\sum_{t=0}^{T-2}\mathbb{E}(\|\nabla f(x_{t+1}) \|^2+ \|\nabla f(x_{t})\|^2). \nonumber
\end{align}
Recalling $\eta = \frac{1}{192(d+2)\sqrt{T}\max(L_0,L_1)}$ and $\delta^2 =\frac{1}{\sqrt{T}}$, we have $\alpha<\frac{1}{2}$, and thus the above inequality yields
\begin{align}
& \mathbb{E}\|x_T - x^*\|^2  \leq \|x_0 - x^*\|^2 - 2\eta \sum_{t=0}^{T-1} \mathbb{E}\big(f(x_t) - f(x^*) \big) \nonumber \\
& \quad \quad +2\eta\delta^2 L_1 d T + 2\eta^2\mathbb{E} \|\widetilde g_b(x_0)\|^2 +2 T \eta^2 \beta \nonumber \\
& \quad \quad + 48(d+4)\eta^2\sum_{t=0}^{T-2}\mathbb{E}(\|\nabla f(x_{t+1}) \|^2+ \|\nabla f(x_{t})\|^2). \nonumber
\end{align}
Since the convexity implies that $\frac{1}{2L_1}\|\nabla f(x) \|^2\leq f(x)-f(x^*)$ for any $x$, rearranging the above inequality yields
\begin{align}
& (2-192(d+4)L_1\eta) \frac{1}{T}\sum_{t=0}^{T-1}\mathbb{E}(f(x_t)-f(x^*)) \nonumber \\
& \leq \frac{\|x_0-x^*\|^2}{\eta T}+ 2\delta^2 L_1 d + \frac{2\eta \mathbb{E}\|\widetilde g_b(x_0)\|^2}{T} \nonumber \\
& +24\eta L^2_1\delta^2(d+6)^3 + \frac{96\eta (d+4)\sigma_g^2}{b}+ \frac{16\eta(d+2)\sigma^2}{\delta^2 b}, \nonumber
\end{align}
which, in conjunction with $\mathbb{E}\|\widetilde g_b(x_0)\|^2\leq M d^2 T$ for certain constant $M>0$, $b=\max\big(\frac{\sigma_g^2}{\sqrt{T}},\sigma^2\big)\sqrt{T}/d$ and recalling that  $\zeta$ is chosen from $0,...,T-1$ uniformly at random, yields
\begin{align}
\mathbb{E}(f(x_\zeta)-f(x^*)) \leq \mathcal{O} \Big( \frac{d}{\sqrt{T}}+ \frac{d^2}{T}\Big). \nonumber 
\end{align}
Then, to achieve an $\epsilon$-accurate solution, i.e.,  $\mathbb{E}(f(x_\zeta)-f(x^*))\leq \epsilon$, we need $T=\mathcal{O}(d^2\epsilon^{-2})$, and thus the corresponding query complexity is given by 
\begin{align}
Tb \leq \mathcal{O}(\sigma^2d^2\epsilon^{-3}+ \sigma_g^2d\epsilon^{-2}). \nonumber
\end{align}
\end{pf}

\end{document}